\documentclass[12pt]{amsart}
\usepackage{amssymb}
\usepackage{mathtools}
\usepackage{amsbsy}
\usepackage{amsfonts}
\usepackage{latexsym}
\usepackage{amsopn}
\usepackage{amstext}
\usepackage{cite}
\usepackage{amsxtra}
\usepackage{euscript}
\usepackage{amscd}
\usepackage{bm}
\usepackage{mathabx}
\usepackage{mathrsfs}
\usepackage{abraces}
\usepackage{url}
\usepackage{mathtools}
\mathtoolsset{showonlyrefs}
\usepackage[dvipsnames]{xcolor}
\usepackage{tikz}

\usepackage{color}
\usepackage{hyperref}

\newcommand{\dist}{\operatorname{dist}}

\newcommand{\RE}{\operatorname{Re}}

\usepackage{bibentry}

\usepackage[english]{babel}
\usepackage{todonotes}
\usepackage{url}

\newcommand{\GL}{\operatorname{GL}}
\newcommand{\SL}{\operatorname{SL}}

\usepackage[top=1in, bottom=1.25in, left=1.25in, right=1.25in]{geometry}

\begin{document}

\newtheorem{thm}{Theorem}
\newtheorem{lem}[thm]{Lemma}
\newtheorem{claim}[thm]{Claim}
\newtheorem{cor}[thm]{Corollary}
\newtheorem{prop}[thm]{Proposition} 
\newtheorem{definition}[thm]{Definition}
\newtheorem{question}[thm]{Open Question}
\newtheorem{conj}[thm]{Conjecture}
\newtheorem{prob}{Problem}

\theoremstyle{remark}
\newtheorem{rem}[thm]{Remark}

\newcommand{\hh}{{{\mathrm h}}}

\numberwithin{equation}{section}
\numberwithin{thm}{section}
\numberwithin{table}{section}
\numberwithin{figure}{section}

\def\sssum{\mathop{\sum\!\sum\!\sum}}
\def\ssum{\mathop{\sum\ldots \sum}}
\def\dsum{\mathop{\sum \, \sum}}
\def\iint{\mathop{\int\ldots \int}}

\newcommand{\diam}{\operatorname{diam}}


\newfont{\teneufm}{eufm10}
\newfont{\seveneufm}{eufm7}
\newfont{\fiveeufm}{eufm5}
%
%
\newfam\eufmfam
     \textfont\eufmfam=\teneufm
\scriptfont\eufmfam=\seveneufm
     \scriptscriptfont\eufmfam=\fiveeufm
%
%
\def\frak#1{{\fam\eufmfam\relax#1}}

\newcommand{\bflambda}{{\boldsymbol{\lambda}}}
\newcommand{\bfmu}{{\boldsymbol{\mu}}}
\newcommand{\bfxi}{{\boldsymbol{\eta}}}
\newcommand{\bfrho}{{\boldsymbol{\rho}}}

\def\eps{\varepsilon}

\def\fK{{\mathcal K}}
\def\fI{\mathfrak I}
\def\fT{\mathfrak{T}}
\def\fL{\mathfrak L}
\def\fR{\mathfrak R}

\def\fA{{\mathfrak A}}
\def\fB{{\mathfrak B}}
\def\fC{{\mathfrak C}}
\def\fD{{\mathfrak D}}
\def\fM{{\mathfrak M}}
\def\fQ{{\mathfrak  Q}}
\def\fS{{\mathfrak  S}}
\def\fU{{\mathfrak U}}
\def\fW{{\mathfrak W}}

\def\T {\mathsf {T}}
\def\Tor{\mathsf{T}_d}
\def\Tore{\widetilde{\mathrm{T}}_{d} }

\def\sM {\mathsf {M}}

\def\ss{\mathsf {s}}

\def\Kmnd{\cK_d(m,n)}
\def\Kmnp{\cK_p(m,n)}
\def\Kmnq{\cK_q(m,n)}

 \def \Aprime{C} 
\def \balpha{\bm{\alpha}}
\def \bbeta{\bm{\beta}}
\def \bgamma{\bm{\gamma}}
\def \bdelta{\bm{\delta}}
\def \bzeta{\bm{\zeta}}
\def \blambda{\bm{\lambda}}
\def \bchi{\bm{\chi}}
\def \bxi{\bm{\xi}}
\def \bphi{\bm{\varphi}}
\def \bpsi{\bm{\psi}}
\def \bnu{\bm{\nu}}
\def \bomega{\bm{\omega}}

\def \bell{\bm{\ell}}

\def\vec#1{\mathbf{#1}}

\newcommand{\abs}[1]{\left| #1 \right|}

\def\dimens{d}
\def\Zq{\mathbb{Z}_q}
\def\Zqx{\mathbb{Z}_q^*}
\def\Zd{\mathbb{Z}_d}
\def\Zdx{\mathbb{Z}_d^*}
\def\Zf{\mathbb{Z}_f}
\def\Zfx{\mathbb{Z}_f^*}
\def\Zp{\mathbb{Z}_p}
\def\Zpx{\mathbb{Z}_p^*}
\def\cM{\mathcal M}
\def\cE{\mathcal E}
\def\cH{\mathcal H}

\def\sfB{\mathrm {B}}
\def\sfC{\mathsf {C}}
\def\sfU{\mathrm {U}}
\def\L{\mathsf {L}}
\def\FF{\mathsf {F}}

\def\sE {\mathscr{E}}
\def\sK {\mathscr{K}}
\def\sR {\mathscr{R}}
\def\sS {\mathscr{S}}

\def\cA{{\mathcal A}}
\def\cB{{\mathcal B}}
\def\cC{{\mathcal C}}
\def\cD{{\mathcal D}}
\def\cE{{\mathcal E}}
\def\cF{{\mathcal F}}
\def\cG{{\mathcal G}}
\def\cH{{\mathcal H}}
\def\cI{{\mathcal I}}
\def\cJ{{\mathcal J}}
\def\cK{S}
\def\cL{{\mathcal L}}
\def\cM{{\mathcal M}}
\def\cO{{\mathcal O}}
\def\cP{{\mathcal P}}
\def\cQ{{\mathcal Q}}
\def\cR{{\mathcal R}}
\def\cS{{\mathcal S}}
\def\cT{{\mathcal T}}
\def\cU{{\mathcal U}}
\def\cV{{\mathcal V}}
\def\cW{{\mathcal W}}
\def\cX{{\mathcal X}}
\def\cY{{\mathcal Y}}
\def\cZ{{\mathcal Z}}
\newcommand{\rmod}[1]{\: \mbox{mod} \: #1}

\def\cg{{\mathcal g}}

\def\vy{\mathbf y}
\def\vr{\mathbf r}
\def\vx{\mathbf x}
\def\va{\mathbf a}
\def\vb{\mathbf b}
\def\vc{\mathbf c}
\def\ve{\mathbf e}
\def\vh{\mathbf h}
\def\vk{\mathbf k}
\def\vm{\mathbf m}
\def\vz{\mathbf z}
\def\vu{\mathbf u}
\def\vv{\mathbf v}

\def\e{{{e}}}
\def\ep{{\mathbf{\,e}}_p}
\def\eq{{\mathbf{\,e}}_q}

\def\Tr{{\mathrm{Tr}}}
\def\Nm{{\mathrm{Nm}}}

 \def\SS{{\mathbf{S}}}

\def\lcm{{\mathrm{lcm}}}

 \def\0{{\mathbf{0}}}

\def\({\left(}
\def\){\right)}
\def\fl#1{\left\lfloor#1\right\rfloor}
\def\rf#1{\left\lceil#1\right\rceil}
\def\fl#1{\left\lfloor#1\right\rfloor}
\def\ni#1{\left\lfloor#1\right\rceil}
\def\sumstar#1{\mathop{\sum\vphantom|^{\!\!*}\,}_{#1}}

\def\mand{\qquad \mbox{and} \qquad}

\def\tblue#1{\begin{color}{blue}{{#1}}\end{color}}







\mathsurround=1pt

\def\bfdefault{b}

\def \F{{\mathbb F}}
\def \K{{\mathbb K}}
\def \N{{\mathbb N}}
\def \Z{{\mathbb Z}}
\def \P{{\mathbb P}}
\def \Q{{\mathbb Q}}
\def \R{{\mathbb R}}
\def \C{{\mathbb C}}
\def\Fp{\F_p}
\def \fp{\Fp^*}

 \def \xbar{\overline x}

 \title[Sums of Kloosterman sums and the discrepancy of modular inverses]{Triple sums of Kloosterman sums and the discrepancy of modular inverses}


\author[V. Blomer]{Valentin Blomer}
 \address{Mathematisches Institut, Universit{\"a}t Bonn, 
Bonn, D-53115 
Germany}
 \email{blomer@math.uni-bonn.de}
 

 \author[M.~S.~Risager]{Morten S. Risager}
 \address{Department of Mathematical Sciences, 
  Universitetsparken~5,
University of Copenhagen, 
2100 Copenhagen Ø, 
Denmark}
 \email{risager@math.ku.dk}

 \author[I.~E.~Shparlinski]{Igor E. Shparlinski}
 \address{School of Mathematics and Statistics, University of New South Wales.
 Sydney, NSW 2052, Australia}
 \email{igor.shparlinski@unsw.edu.au}
\date{\today}

\begin{abstract} 
We investigate the distribution of modular inverses modulo positive integers $c$ in a large interval. We  provide upper and lower bounds  for their   box, ball and  isotropic discrepancy, thereby exhibiting some deviations from random point sets. 
  The analysis is based, among other things, on a new bound for a triple sum of Kloosterman sums.   
 \end{abstract}


\maketitle


\section{Introduction} 
\subsection{Motivation and set-up}

 For  a given integer $c \ge 1 $, the distribution of the   points 
\begin{equation}
\label{eq: Set ab}
 \left\{\(\frac{a}{c}, \frac{b}{c}\):~ \begin{matrix*}[l]
  ab \equiv 1 \bmod c,\\ 1 \le a,b \leq c\end{matrix*}\right\}
\subseteq	(\mathbb{R}/\mathbb{Z})^2
\end{equation}
has been 
 studied in a large number of works; see~\cite{Shp2} for a survey 
and also more recent works~\cite{Baier,BrHay,Chan,GarShp,Hump,Ust}.  In this paper we vary $c$ as well and  
look at the distributional properties of the set of pairs
 \begin{equation}
\label{eq: Set abc}
\cS(X) =  \left\{\(\frac{a}{c}, \frac{b}{c}\):~\begin{matrix*}[l]ab \equiv 1 \bmod c, \\
1 \le a,b \leq c \le X
\end{matrix*}\right\}
\subseteq	(\mathbb{R}/\mathbb{Z})^2,
\end{equation}
which corresponds to   the union of the sets~\eqref{eq: Set ab}  over all $c\leq X$.  As $X\to \infty$, for the  cardinality $N(X) = \# \cS(X)$ we have the classical asymptotic formula
\begin{equation}
\label{eq: NX-total}
N(X) = \sum_{c \le X} \varphi(c) =\( \frac{3}{\pi^2} + o(1)\) X^2,
\end{equation} 
 where, as usual, $\varphi(k)$ denotes the Euler function;  
 see~\cite[Satz~1, p.~144]{Walf} and~\cite{Liu}  for explicit error estimates.  

The distribution of points from $\cS(X)$ has been  considered by Selberg~\cite{Selb} and later by
Good~\cite[p.~119--120]{Good}, 
although this seems to not be so widely known. They both use spectral theory of automorphic forms to establish  equidistribution  via  the Weyl criterion,   but they do not state explicit error terms.  Our main results below provide in particular quantitative forms of the equidistribution result of Selberg  
and Good.  

The connection to the spectral theory of automorphic forms comes from the  following observation: writing the congruence $ab \equiv 1 \bmod c$ as
$ab-cd=1$, the set $\cS(X)$ can be interpreted  
as a set of double cosets in  $\SL_2(\Z)$:  
there is a bijection
\[
  \begin{array}{ccc}
\{\gamma\in \Gamma_{\!\infty}\backslash\SL_2(\Z)\slash\Gamma_{\!\infty}:~  0< c \leq X\}&\to &  \cS(X), \\
\gamma = \begin{pmatrix}
  a&d\\c&b
\end{pmatrix} 
& \mapsto & \displaystyle \(\frac{a}{c}, \frac{b}{c}\) ,
\end{array}
\]
where as usual 
$
\Gamma_{\!\infty}=\left\{\pm\left(\begin{smallmatrix}1&n\\0&1 \end{smallmatrix}\right), :~ n\in\Z\right\}$.

Another popular interpretation of such results is in terms of the distribution of solutions 
of linear equations; see~\cite{BetCha,DinSin,Dolg,Fuji,Rieg, Shp1}.

Our results are based on a new bound of triple sums of Kloosterman sums, which we believe 
is of independent interest, and also on some  
techniques in distribution theory, 
including those of Barton,  Montgomery and Vaaler~\cite{BMV} and 
of Schmidt~\cite{Schm}, which we revise and present with a self-contained argument.

\subsection{Triple sums of Kloosterman sums}\label{sec12}

Fourier-analytic techniques to investigate the distributional properties of the sets~\eqref{eq: Set ab} or~\eqref{eq: Set abc} lead to Kloosterman sums
\[ 
\cK(m,n;c) = \sum_{\substack{a,b=1\\
ab \equiv 1 \bmod c}}^{c} \e\(\frac{ma + nb}{c}\) , 
\]
where $\e(z) = \exp(2\pi i z)$. In fact, the vast majority of works has been 
relying on the pointwise bound
\begin{equation}
\label{eq: Kloost}
|\cK(m,n;c)| \le \gcd(m,n,c)^{1/2} c^{1/2} \tau(c),
\end{equation}
 where $\tau$ is the divisor function. This is the 
celebrated Weil bound; see for example~\cite[Corollary~11.12]{IwKow}.

 A major ingredient in our analysis  is a bound for the following
triple sum of Kloosterman sums.  
For integers $M, N\ge 1$ and real $X\ge 1$ we define
\[
\fK(M,N;X) =  \sum_{M \leq |m| < 2M} \sum_{N \leq |n| <  2N} \left|\sum_{c  \le X}  \cK(m, n;c)\right|.
\]

There are uniform individual bounds available for the innermost sum over $c$, due to 
Sarnak and Tsimerman~\cite[Theorem~2]{SarTsim}
for $mn> 0$ and  K{\i}ral~\cite[Theorem~2]{Kir} for $mn<0$.  These bounds  imply that
\begin{equation}
\label{eq: KloostAver}
\left|\sum_{c \le X}  \cK(m,n;c)  \right|
\le \( |mn|^{1/4} X+ |mn|^\vartheta X^{7/6} \) (|mn|X)^{o(1)} ,  
\end{equation}
as $|mn|X \to \infty$, where  $\vartheta$ is the best known 
exponent towards the Ramanujan--Petersson conjecture.  
Using~\eqref{eq: KloostAver}, one immediately derives 
\[
\fK(M,N;X)  \le \((MN)^{5/4}X+(MN)^{1+\vartheta} X^{7/6}\)(MNX)^{o(1)}, 
\]
We  obtain a substantially stronger (even under the assumption $\vartheta = 0$) 
bound on average over $n$ and $m$, which does not depend on $\vartheta$ 
and which we believe to be of independent interest. 

\begin{thm}
\label{thm:Triple Sum} 
Let $M, N, X \ge 1$ be  real numbers. Then
\[
\fK(M,N;X)  \le \(MNX+(MN)^{2/3} X^{7/6}\)(MNX)^{o(1)}, 
\]
as $MNX \to \infty$.  
\end{thm}

\begin{rem}\label{rem:sumRaman} Using the trivial bound  
$|\cK(0,n;c)| \le \gcd(c,n)$ for Ramanujan sums, the statement remains true for $M=0$ if we replace $M$ on the right-hand side with $M+1$.   
\end{rem}

The first term $MNX$ is consistent with square-root cancellation in the $c$-sum. 
The second term is an artefact of the sharp cut-off (with respect to $c$) and disappears when we insert a smooth weight $W(c/X)$. For $M=N = 1$ it matches the classical bound of Kuznetsov~\cite[Theorem~3]{Ku}. It is dominated by the first term as soon as $MN > X^{1/2}$.

For comparison, we note that the (modified) {\it Selberg--Linnik conjecture\/} states that
\begin{equation}\label{Selberg-Linnik}
\left|\sum_{c \le X}   \cK(m,n;c)  \right|
\le X(|mn|X)^{o(1)} ;
\end{equation}
see, for example,~\cite[Equation~(7)]{SarTsim}. This is however completely out of reach using current methods.
 
 
Theorem~\ref{thm:Triple Sum} is a simple consequence; see Section~\ref{sec:prelim}, of the following  stronger result for the second moment
\[
\fK^{(2)}( N;X) =  \sum_{N \leq |n| <  2N} \left|\sum_{c  \le X}  \cK(n, 1;c)\right|^2.
\]

\begin{thm}
\label{thm:second} 
 Let $N, X \ge 1$ be   real numbers. Then
\[
\fK^{(2)}(N;X)  \le \(NX^2+N^{1/3} X^{7/3}\)(NX)^{o(1)}, 
\]
as $NX \to \infty$.  
\end{thm}

Again the first term is best possible and dominates the second as soon as $N > X^{1/2}$. Historically, a first attempt to bound a second moment of sums of Kloosterman sums, written before the spectral large sieve of Deshouillers--Iwaniec was available and  in several aspects less useful, can be found in~\cite[Section~3]{PY}. 


We apply Theorem~\ref{thm:Triple Sum} to 
estimate 
 how well $\cS(X)$ equidistributes on various types of sets.   
  It is reasonable to expect that the set $\cS(X)$ behaves by and large like a random set of $N(X)$ points in $(\R/\Z)^2$, although on a small scale it has some interesting arithmetic features; see Figure~\ref{scatter-plot}.

  \begin{figure}[ht]
    \begin{center}
        \includegraphics[scale=0.8]{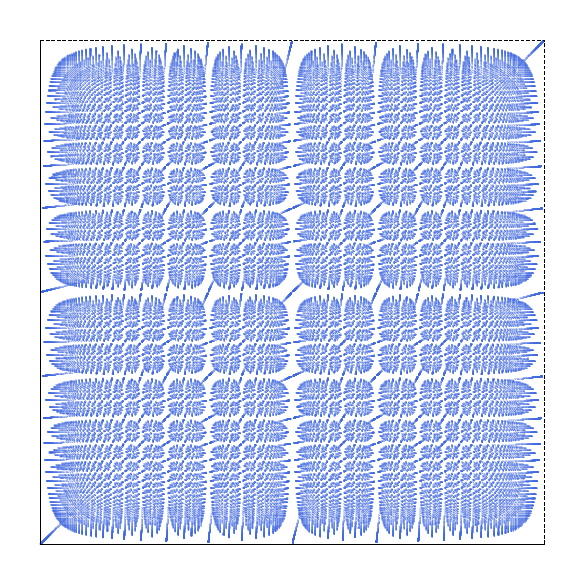}
    \caption{ Modular inverses: $X=600$, $N(X)=109500$.}
    \label{scatter-plot}
    \end{center}
    \end{figure}
  
   By~\eqref{eq: NX-total}, the average spacing between two points in $\cS(X)$ is about $1/X$, so most optimistically one might hope to say something about the distribution of points in sets of volume a little larger than $1/X^2$. Clearly, for smaller sets 
we cannot expect 
any reasonable statistics. However, already for boxes of volume $1/X$ some irregularities occur because under the hyperbola $\{(x, y) \in [0, 1)^2 :~ xy < 1/X\}$ the only points in $\cS(X)$ are 
\begin{equation}\label{eq: 1/c 1/c}
   (x, y) = (1/c, 1/c),  \qquad \sqrt{X} < c \leq X. 
   \end{equation}  
   It follows that
\begin{itemize}
\item the box $[0, 1] \times [0, 1/(2X)]$ contains no point from $\cS(X)$;   
\item the box $[2/\sqrt{X}, 3/\sqrt{X}] \times [0.2/ \sqrt{X}, 0.25/ \sqrt{X}]$ contains no point from $\cS(X)$;  
\item the box $[0,  1/\sqrt{X}] \times [0, 1/\sqrt{X}]$ contains exactly the $(1 + o(1))X$ points~\eqref{eq: 1/c 1/c} 
while by~\eqref{eq: NX-total} the expected value is only $(3/\pi^2 + o(1))X$.  
\end{itemize}
Similar phenomena occur at other rational points with small denominators.

From this discussion we conclude the following:
\begin{itemize}
\item  The smallest distance between two points in $\cS(X)$ is of size $1/X^2$, namely between the points 
\[(1/[X], 1/[X]) \quad \text{and} \quad (1/([X] -1), 1/([X]-1)).
\] 
\item There are discs of radius $r \gg 1/\sqrt{X}$ without points in  $\cS(X)$, while using the bound~\eqref{eq:DiscrSmooth} below, we see that this  is tight since for any fixed $\varepsilon > 0$ and sufficiently 
large $X$, any disc of radius $r \ge  X^{-1/2+ \varepsilon}$ contains a point from  $\cS(X)$
(in fact, at least $cr^2$ such points for some absolute constant $c>0$). 
\end{itemize}

 \subsection{Discrepancy} 
 
\subsubsection{General definitions}

An important quantity  when studying the distribution  of a finite set $S$ in some ambient space $\cT$ is the discrepancy. Let $\cA$ be a set of measurable subsets of $\cT$. Then the discrepancy $\Delta(S, \cA)$ of $S$ with respect to $\cA$ is
\[
\Delta(S, \cA) = \underset{A \in \cA}{\sup} \left| \frac{\# (S \cap A)}{\# S} - \mu(A) \right|.
\]
If $S = \cS(X)$, where $\cS(X)$ is given by~\eqref{eq: Set abc}, we write  $\Delta(X, \cA)$ instead of $\Delta(\cS(X), \cA)$ for notational simplicity.

Typical choices for $\cA$ are the subset of boxes or the subset of balls or the subset of convex sets inside of some translation of the unit square 
$[0,1)^2$, which are then mapped to $ \(\R/ \Z\)^2$ under the natural reduction modulo $1$.  With this in mind, let $\cB$ be the set of boxes 
\begin{equation}\label{eq: Box B}
   B= [\xi, \xi+ \alpha] \times [\zeta, \zeta + \beta] \subseteq \(\R/ \Z\)^2
   \end{equation}
with $0\leq \alpha, \beta < 1$, let $\cD$  be the set of (injective) discs 
\begin{equation}\label{eq: Disc D}
   D = \{\bm \xi :~\| \bm \xi - \bm \omega  \| < R \}\subseteq \(\R/ \Z\)^2
   \end{equation}
for $0 \leq R < 1/2$, and let $\cC$ the  set of convex subsets in 
\[
\sfU = [0, 1)^2, 
\]
 which we identify with $ \(\R/ \Z\)^2$.   Note that a convex subset of any translate of  $\sfU$, after mapping to 
 $ \(\R/ \Z\)^2$ can be partitioned into finitely many convex subsets of the above type. Hence, it is enough to study the discrepancy with respect to the above class of convex sets. The discrepancy with respect to convex sets is often called isotropic discrepancy.  We   refer to~\cite{DrTi,KuNi} for historical background on isotropic discrepancy.

In the framework of empirical processes indexed by sets, one would expect that the discrepancy of random point sets with respect to any reasonable class of test sets $\cA$ satisfies the law of iterated logarithm; in particular
\begin{equation}\label{limsup}
 \limsup_{\#S \rightarrow \infty}\( \Delta(S, \cA)   \frac{\sqrt{\#S}}{\sqrt{2\log\log \#S}}\) = \frac{1}{2} \quad \text{almost surely}.
 \end{equation}
 In particular this holds for the isotropic discrepancy of random point sets in 
 $\(\R/\Z\)^2$; cf.\  
Philipp~\cite[Section~3]{Phil}.

Available techniques, however, are very sensitive to the class of test sets. For instance, 
 by a special (and somewhat crude) form of the result of Schmidt~\cite[Theorem~2]{Schm},  we   have 
\begin{equation}\label{schmidtbound}
\Delta(X, \cC) \ll \Delta(X, \cB)^{1/2}; 
\end{equation}
see also~\cite[Chapter~2, Theorem~1.6]{KuNi}.

\subsubsection{Lower bounds}
From the end of  Section~\ref{sec12} we conclude
\[
\Delta(X, \cB) \gg 1/X \mand \Delta(X, \cD) \gg 1/X.
\]
A slightly sharper lower bound can be obtained for $\Delta(X, \cC)$.  

\begin{thm}\label{thm:lower-convex} We have
\[\Delta(X, \cC) \gg \frac{\log X}{X}.
\]
\end{thm}

\begin{proof} We observe that there are 
$O(X)$ points from    $\cS(X)$  under the hyperbola $\{(x, y) \in \sfU^2 :~ xy < 1/X\}$.  Therefore, the convex domain 
$\left\{(x, y) \in \sfU^2 :~ xy  \geq 1/X\right\}$ contains $N(X) + O(X)$ points from $\cS(X)$,  while its area is 
$1 - X^{-1} \log X  + O\(X^{-1}\)$. We conclude that 
\[
\Delta(X, \cC) \ge \left | \frac{N(X) + O(X)}{N(X)} -  1  + \frac{ \log X  + O(1)}{X}\right | = \frac{\log X + O(1)}{X}, 
\]which in fact is a little stronger than the claimed bound.
\end{proof}

\begin{rem} We emphasize that  the lower bound in Theorem~\ref{thm:lower-convex} is slightly larger than what we would expect with probability one from a random point set 
by~\eqref{limsup}. In particular, the set $\cS(X)$ has some arithmetic structures that, at least on a small scale, lead to deviations from randomness. 
We refer again to  Figure~\ref{scatter-plot} whose distinctive cellular structure appears to be shaped by underlying arithmetic properties; the cells sit between lines with rational coordinates with small denominators.
See also~\cite{BBRR} and~\cite{GarShp} for such deviations in slightly different arithmetic contexts. 
\end{rem}

\subsubsection{Upper bounds}

The bound~\eqref{schmidtbound} indicates that our techniques to analyze  the discrepancy are quite sensitive to the choice of $\cA$, which becomes in particular apparent if Fourier-analytic methods are applied: in this case the bounds  depend on regularity properties of the Fourier transform of the characteristic functions of $A \in \cA$. The Fourier transform of boxes is  in the $\ell^{1+\varepsilon}$-space for any fixed $\varepsilon>0$,  which leads to fairly strong bounds.  
The Fourier transform of two-dimensional balls is only in 
the $\ell^{4/3 + \varepsilon}$-space; cf.~\cite[Equations~(2.3) and~(2.6)]{Hump}), 
 so one may expect slightly weaker bounds in this case. For general convex sets we have much less control on the Fourier transform of the corresponding characteristic functions, and we  study the isotropic discrepancy   by an approximation argument to be described later.

\begin{thm}
\label{thm:Discr}
We have 
\[
\Delta(X, \cB)  \le X^{-5/6 + o(1)}, 
\]
as $X\to \infty$.  
\end{thm}

The exponent $-5/6$ in Theorem~\ref{thm:Discr} is an artefact of the sharp cut-off  $c\leq X$ in the definition~\eqref{eq: Set abc} of $\cS(X)$. If we study a ``weighted discrepancy'', which includes a smooth weight $W(c/X)$, our argument gives the essentially best possible 
bound 
\begin{equation}
\label{eq:DiscrSmooth}
\Delta_{\text{smooth}}(X, \cB) \leq X^{-1 + o(1)}.
  \end{equation}
    The same bound would follow also for sharp cut-offs from the modified Selberg--Linnik 
conjecture~\eqref{Selberg-Linnik}, which we reiterate is out of reach with current technology.

For small boxes we can refine  Theorem~\ref{thm:Discr} 
using 
techniques of 
Barton, Montgomery and Vaaler~\cite{BMV} (see Lemma~\ref{lem:BMW-corollary} below),  which appear to  be less known than they deserve.

\begin{thm}
\label{thm:SmallBox}
For a box $B$ as in~\eqref{eq: Box B}  such that 
\begin{equation}
\label{eq: alpha/beta}
 \mu(B) >  X^{-1} 
  \end{equation}
we have 
\[
\left| \frac{\#(B\cap \cS(X))}{N(X)} - \mu(B)  \right|  \le \mu(B)^{2/3}  X^{-1/3+o(1)}, 
\]
as $X\to \infty$. 
\end{thm} 

Theorem~\ref{thm:SmallBox} improves Theorem~\ref{thm:Discr}  
 for $\mu(B) < X^{-3/4+\varepsilon}$, while for large boxes $B$,  
 Theorem~\ref{thm:Discr} is more precise. 

The condition~\eqref{eq: alpha/beta} is introduced so that our argument  reaches its  maximal strength.  Our argument continues to work beyond the 
restrictions~\eqref{eq: alpha/beta} and still produces nontrivial results. On the other hand, as we saw above, general statements about boxes of volume below $1/X$ are only of restricted relevance.

Combining our Theorem~\ref{thm:Triple Sum} with  ideas of Harman~\cite[Theorem~2]{Harm}
 we derive the following bound for the distribution in discs. 

\begin{thm}
\label{thm:SmallDisc}
For a disc $D$ as in~\eqref{eq: Disc D}  of radius $R < 1/2$  we have 
\[
\left| \frac{\#(D\cap \cS(X))}{N(X)} - \pi R^2 \right|  \le (X^{-1} + R^{2/3} X^{-2/3 })X^{o(1)}, 
\]
as $X\to \infty$. 
\end{thm}  

The term  $R^{2/3} X^{-2/3}$ in the bound of Theorem~\ref{thm:SmallDisc}  is less than the main term $\pi R^2$ as long as $R > X^{-1/2 + o(1)}$. 
We  see from Theorem~\ref{thm:SmallDisc} that 
\[
\Delta(X, \cD) \le X^{-2/3 + o(1)}.\]


Turning to the isotropic discrepancy, from~\eqref{schmidtbound} and   Theorem~\ref{thm:Discr} we obtain the baseline bound 
$
\Delta(X, \cC)   \le X^{-5/12 + o(1)}, 
$
as $X\to \infty$.  Using the approach of~\cite{Kerr, KerShp, Shp3}, which in turn exploits some ideas
of Schmidt~\cite[Theorem~2]{Schm}, we  establish a stronger bound. 

\begin{thm}
\label{thm:Isotrop Discr}
We have 
\[
\Delta(X, \cC)  \le X^{-11/24+ o(1)}, 
\]
as $X\to \infty$.  
\end{thm}

\section{Preliminaries}

\subsection{Notation and conventions}
For a finite set $\cS$, we use $\# \cS$ to denote its cardinality.  As  usual, the  equivalent 
notations
\[
U = O(V) \quad \Longleftrightarrow \quad U \ll V \quad \Longleftrightarrow \quad V\gg U
\] 
all mean that $|U|\le  C V$ for some positive constant $C$, 
which, unless indicated otherwise, is absolute.   
We also write $U = V^{o(1)}$ 
if, for any given $\varepsilon>0$,  there exist $c(\varepsilon) > 0$ such that   $ |U|\le c(\varepsilon)V^\varepsilon$.
Furthermore, for two quantities $U$ and $V$, depending on $X,M,N$,  we use the notation
\[U \preccurlyeq V \quad \Longleftrightarrow \quad U \ll V(MNX)^{o(1)}.\]

\subsection{Point distribution in boxes  and exponential sums}
\label{sec:discrepancy}    

For the sake of completeness, we present the auxiliary results in full generality  for $\dimens$-dimensional 
finite sets   $S$ in the set $\cB$ of $d$-dimensional boxes  
\begin{equation}
\label{eq:Box-B-sDim}
B = [\xi_1, \xi_1 +\alpha_1 ] \times \ldots  \times [\xi_\dimens, \xi_\dimens +\alpha_\dimens
]\subseteq  \(\R/\Z\)^\dimens,
\end{equation}
of volume $\mu\(B\) = \alpha_1 \cdots \alpha_\dimens$. For our application we only need the case $d=2$.

One of the basic tools used to study uniformity of
distribution is the celebrated
 Koksma--Sz\"usz inequality~\cite{Kok,Sz} 
(see also~\cite[Theorem~1.21]{DrTi}), which generalises the Erd\H{o}s--Tur\'an inequality to general dimensions. Like the Erd\H{o}s--Tur\'an inequality, the Koksma--Sz\"usz inequality links the discrepancy of a sequence of points to certain
exponential sums.

In this section, all implied constants may depend only on the dimension $\dimens$.

\begin{lem}\label{lem:K-S}
For any integer $M \ge 1$, we have
\[
 \Delta(S, \cB) \ll \frac{1}{M}
+\frac{1}{\#S}\sum_{\substack{\vm \in \mathbb{Z}^d\\ 0<\|\vm\|\le M}}\frac{1}{r(\vm)}
\left| \sum_{s\in S} \e{(\langle \vm, s\rangle)}\right|,
\]
where $ \langle \vm, s\rangle$ denotes the inner product in $\R^\dimens$, 
\[
\|\vm\|=  \max_{j=1, \ldots, \dimens}\left|m_j\right| \mand r(\vm) =  \prod_{j=1}^\dimens \(\left|m_j\right| + 1\). 
\]
\end{lem}  

Next we formulate~\cite[Theorems~1 and~3]{BMV}, which give upper 
and lower bounds for
$\#(S \cap B) $  for individual boxes $B$. 
 To simplify the 
exposition, we present these bounds without explicit constants and 
with some other simplifying assumptions.  

\begin{lem}\label{lem:BMW-corollary} Let   
 $S \subseteq  
\(\R/\Z\)^\dimens$  be a finite set  and let $B$  be a box as 
  in~\eqref{eq:Box-B-sDim}. 
 Assume that $L_i\in \N$ satisfies $\alpha_i L_i\geq 2$ for $i=1,\ldots,\dimens$. Then
\[
  \#(S \cap B)  = \mu\(B\) \#S\(1 + O(\max\{E, E^2\}\)
\]
  where 
\[
E  = \sum_{i=1}^\dimens\alpha_i^{-1} L_i ^{-1}
+ \frac{1}{\#S}
\sum_{\substack{\vm \in \cL\\ \vm \ne \mathbf 0 }}
\left| \sum_{s \in S} \e\( \langle \vm, s\rangle\)\right|,
\]
in which the sum is taken over all non-zero integer vectors $\vm$ 
in the box $\cL =  [-L_1, L_1] \times \cdots\times  [-L_\dimens, L_\dimens]$. 
\end{lem}

\begin{proof} We first consider the case when there are no points of $S$ on the boundary of $B$, and where 
\begin{equation}
\label{eq:IntegrCond}
\alpha_iL_i \in \N, \qquad i=1, \ldots, \dimens, 
\end{equation}
to make it closer to the setting of~\cite[Theorems~1 and~3]{BMV}. 

We merely explain deviations from the setting 
of~\cite[Theorems~1 and~3]{BMV}. First we replaced $L_i+1$ with $L_i$, which means we just add more 
terms over which we average absolute values of the exponential sums in $E$. 
Now the lower bound 
\[
  \#(S \cap B)  \ge \mu\(B\) \#S\(1 + O(E)\)
\]
follows from~\cite[Theorem~1]{BMV}.

To derive the upper bound for $  \#(S \cap B)$ 
we note that we can assume 
$\alpha_i^{-1} L_i ^{-1} < 1/2$, $i=1, \ldots, \dimens$,  as otherwise the result is trivial. 
Hence, we have the condition~\cite[Equation~(1.10)]{BMV} with 
\[
\delta = \prod_{i=1}^\dimens (1 + \alpha_i^{-1} L_i ^{-1}) - 1 
\ll  \sum_{i=1}^{\dimens}\alpha_i^{-1} L_i ^{-1}
\]
and the bound 
\[
  \#(S \cap B) \le \mu\(B\) \#S\(1 + O(E)\)
\]
follows from~\cite[Theorem~3]{BMV}. 
Hence, under the condition~\eqref{eq:IntegrCond} we have 
\begin{equation}
\label{eq:Bound T-IntCond}
  \#(S \cap B) \le \mu\(B\) \#S\(1 + O(E)\). 
\end{equation}

To drop the condition~\eqref{eq:IntegrCond}, we use a simple approximation argument.  We want to construct sets $B^-\subseteq B\subseteq B^+$ approximating $B$ and satisfying the integrality conditions as in~\eqref{eq:IntegrCond}.  We do this as follows: 
  Choose an integer $2\leq b_i\leq L_i$ such that
\[\frac{b_i}{L_i}\leq \alpha_i<\frac{b_i+1}{L_i},
\] and define
\[ 
\alpha_i^-=\frac{b_i-1}{L_i}, \mand \alpha_i^+=\min\left\{\frac{b_i+2}{L_i},1\right\} .
\]
Then we let
\[
 B^\pm = [\xi_1^\pm, \xi_1^\pm +\alpha_1^\pm ] \times \ldots  \times 
  [\xi_\dimens^\pm, \xi_\dimens^\pm +\alpha_\dimens^\pm ]  , 
\]
where $\xi_i^-\geq\xi_i$, $\xi_i^+\leq\xi_i$ are chosen such that $|\xi_i-\xi_i^\pm|\leq (2L_i)^{-1}$ and such that the boundary of $B^\pm$ contains no points from $S$. Since we only need to avoid finitely many points this is certainly  possible.  Then $B^-\subseteq B\subseteq B^+$ and 
the boxes $B^\pm$  both satisfy the integrality conditions as   in~\eqref{eq:IntegrCond}, with the integers $L_1,\ldots,L_d$. Also, we have $|\alpha_i^+-\alpha_i^-|\leq 3/L_i$, so 
\begin{align*}
  |B^\pm|  &=  \prod_{i=1}^\dimens \(\alpha_i + O (L_i ^{-1})\)
 =  \mu\(B\) \prod_{i=1}^\dimens \(1 + O \(\frac{1}{ \alpha_i  L_i }\)\)\\ &=\mu\(B\) \( 1 + O\(  \sum_{i=1}^\dimens\alpha_i^{-1} L_i ^{-1}\)\)=\mu\(B\) \( 1 + O\( E\)\).
  \end{align*}
 Furthermore, 
\[\sum_{i=1}^d\frac{1}{\alpha_i^\pm L_i}\leq 2\sum_{i=1}^d\frac{1}{\alpha_i L_i},
\]
so defining the corresponding quantity $E^{\pm}$ for the approximate boxes $B^{\pm}$, we have   $E^{\pm} =O(E)$. 
Clearly, 
\[ \#(S \cap B^-)\leq  \#(S \cap B)\leq \#(S \cap B^+).
\]
Since~\eqref{eq:Bound T-IntCond} gives 
\[\#(S \cap B^{\pm})=|B^\pm|\#S\(1 + O(E^{\pm})\)=\mu\(B\) \#S\(1 + O(E)\)^2
\] we get the result.
\end{proof}

\begin{rem} Clearly, for $E \le 1$, which is certainly the most interesting case,  
the error term in the bound of Lemma~\ref{lem:BMW-corollary} simplifies as $O(E)$. 
However, it is conceivable that in some scenarios we have $E > 1$, in which case  it may
still be used as an upper bound 
(rather than a genuine asymptotic formula). 
\end{rem}

\subsection{Point distribution in balls and exponential sums}
\label{sec:disc}    

For the following connection between the number elements of a  finite sequence $ S\subseteq \(\R/\Z\)^2$ in a given disc 
and exponential sums (similar to Lemmas~\ref{lem:K-S} and~\ref{lem:BMW-corollary}) 
is given by Harman~\cite[Theorem~2]{Harm} (as before we formulate it in an arbitrary 
dimension $d$, while we only use it for $d=2$).

\begin{lem}\label{lem:Harm} Let   
 $S \subseteq \(\R/\Z\)^d$  be a finite set  and let $D$  be a ball
\[
  D = \{\bm \xi :~\| \bm \xi - \bm \omega  \| < R \}\subseteq \(\R/ \Z\)^d
 \]
 of radius $R < 1/2$. 
 For any  real $L\ge 1$ we have
\[
  \#(S \cap D)   = \mu\(D\) \#S\(1 + O(E)\), 
\]
  where 
\[
E =  \frac{R^{d-1}}{ L} + \frac{1}{L^d} + \frac{1}{ \#S}  
 \sum_{\substack{\vm\in \Z^2\\0< \|\vm\|_2 \le L} }
 \(\frac{1}{L^d} + \min\left\{R^d, \frac{R^{(d-1)/2}}{\|\vm\|_2^{(d+1)/2}}\right\} \)
 \left| \sum_{s \in S} \e\( \langle \vm, s\rangle\)\right|
\]
and   $\|\vm\|_2 $ denotes the Euclidean norm of $\vm$. 
\end{lem} 

\begin{rem}\label{rem:Humph} We note that an alternative treatment of the ball discrepancy 
is provided by Humphries~\cite{Hump}. This was used in an earlier version of this paper and leads to the same exponents in Theorem~\ref{thm:SmallDisc}. As the referee remarked, the present approach,is slightly shorter. 
\end{rem}

\section{Some geometric arguments}

\subsection{Approximation of well-shaped sets by a union of squares}
\label{sec:approx-alternative}
The following definitions and constructions work in any dimension
and are adapted from~\cite[Section~3]{Schm} and~\cite[Section~2]{KerShp}. However, for simplicity we restrict ourselves to dimension $2$. 

Let $\|\vec{v}\|_2$ denote the Euclidean norm of $\vec{v}\in\R^2$, and define the distance between $\vec{u}\in \R^2$ and any $\Omega\subseteq \R^2$  to be
\[
  \dist(\vec{u},\Omega) = \inf_{\vec{w} \in\Omega}
  \|\vec{u} - \vec{w}\|_2.
\] 

Let 
\[
\sfB=\{\vec{v}\in\R^2:~\|\vec{v}\|_2\leq 1\} \mand\sfU=[0,1)^2
\]
be the closed unit ball and the half-open unit square, respectively. 
Consider a subset $\Omega\subseteq \sfU$. For any $\varepsilon>0$ we define the $\varepsilon$-extended set $\Omega_{\varepsilon}$, and the $\varepsilon$-restricted set $\Omega_{-\varepsilon}$ by 
\begin{align*}
 &\Omega_{\varepsilon}=\{\vec u\in \sfU :~\dist(\vec{u},\Omega) < \varepsilon\},\\
 & \Omega_{-\varepsilon}=\{\vec u\in \Omega :~\dist(\vec{u},\sfU\backslash\Omega) \geq  \varepsilon\}.
\end{align*} 
We then define the sets
\begin{align*}
&\Omega_\varepsilon^{+}  =\Omega_{\varepsilon}\backslash \Omega=\{\vec u\in \sfU\backslash \Omega :~\dist(\vec{u},\Omega) < \varepsilon\}
, \\ 
&\Omega_\varepsilon^{-} =\Omega\backslash \Omega_{-\varepsilon} =\{\vec u\in \Omega :~\dist(\vec{u},\sfU\backslash\Omega) < \varepsilon\}.
\end{align*} 
The  Lebesgue measure of these sets, $\mu(\Omega_\varepsilon^{\pm})$, quantifies how much $\Omega$ has been extended or restricted by $\Omega_{\pm\varepsilon}$, respectively.
For a given $\eta>0$ we say that a set $\Omega\subseteq \sfU$ is \emph{$\eta$-well-shaped} if for all $\varepsilon>0$ we have 
\[
\mu\(\Omega_\varepsilon^{\pm }\)  \le \eta \varepsilon. 
\] 
There is another  
subset of  $\Omega$ that is useful. This is defined by
\[
  \widetilde \Omega_{-\varepsilon}=\{\vec u\in \Omega :~\vec u+\varepsilon \sfB \subseteq \Omega \}.
\] 
  One shows that 
  \begin{equation}\label{eq:stacking-of-restrictions}
  \widetilde\Omega_{-\varepsilon}\subseteq \Omega_{-\varepsilon}\subseteq  
  \widetilde\Omega_{-\varepsilon_0}\quad \textrm{ if }0<\varepsilon_0<\varepsilon.\end{equation}

We now explain how to approximate a well-shaped set $\Omega$ by a union of squares:
Fix irrational $\alpha, \beta \in \R\backslash \Q$. For positive $k$, let $\cQ(k)$ be the set of squares $\Gamma$ of the form
\[
\Gamma_{u,v,k}=\left[\alpha+\frac{u}{k},\alpha+\frac{u+1}{k}\right]\times \left[\beta+ \frac{v}{k},\beta+ \frac{v+1}{k}\right]
\subseteq \R^2,
\]
where $u, v \in \Z$. The role of the irrationality condition is to ensure that a rational point is  always in the interior of the squares in $\cQ(k)$ and therefore in particular in precisely one square.
Furthermore, let 
\[
\cQ_\Omega(k)=\{\Gamma\in \cQ(k):~\Gamma\subseteq \Omega\}  
\]
be the set of squares in $\cQ(k)$ that are contained in $\Omega$.
We claim that for any $\varepsilon>\sqrt{2}/k$  we have 
\begin{equation}\label{eq: inclusions}
 \Omega_{-\varepsilon}\subseteq \widetilde \Omega_{-\frac{\sqrt{2}}{k}}\subseteq \bigcup_{\Gamma\in \cQ_\Omega(k)}\Gamma\subseteq \Omega. 
\end{equation}
The first inclusion follows from~\eqref{eq:stacking-of-restrictions}.
To see that the second inclusion holds, let  $\vec x\in \widetilde \Omega_{-\frac{\sqrt{2}}{k}}$, and choose $u_0, v_0\in\Z$ such that $\vec x\in \Gamma_{k,u_0,v_0}$. For $\vec y\in \Gamma_{k,u_0,v_0}$ we have 
\[
\|\vec x-\vec y\|_2\leq \frac{\sqrt{2}}{k}, 
\] that is, $\vec y\in \vec x + \sqrt{2} k^{-1}\sfB$, so $\vec y\in\Omega$ since $\vec x\in \widetilde \Omega_{-\sqrt{2}/k}$. This proves that $\Gamma_{k,u_0,v_0}\subseteq \Omega$ so $\Gamma_{k,u_0,v_0}\in \cQ_\Omega(k)$, which shows the second inclusion in~\eqref{eq: inclusions}.  The third inclusion is clear from the definition of $\cQ_\Omega(k)$. 

From~\eqref{eq: inclusions} we find, since $\mu(\Gamma)=k^{-2}$ for $\Gamma \in \cQ(k)$, that if $\Omega$ is $\eta$-well-shaped, then for any $\delta>0$ we have
\begin{equation}\label{eq:another-inequality}
  - \eta \frac{(\sqrt{2}+\delta)}{k}+\mu(\Omega)\leq k^{-2} \#\cQ_\Omega(k) \leq \mu(\Omega) 
\end{equation}
and by letting $\delta$ tend to zero the same holds also for $\delta=0$.
We now describe how to pack an $\eta$-well-shaped set $\Omega$ with smaller and smaller squares: Let $\cB_1 = \cQ_\Omega(2)$ and for $i =2,3, \ldots$, let $\cB_i$ be the set of  squares $\Gamma \in \cQ_\Omega(2^i)$ that are not contained in any square from $\cQ_\Omega(2^{i-1})$.

\begin{figure}[ht!]
\begin{tikzpicture}[scale=5]   
\def\drawSquare#1#2#3{
\pgfmathparse{#1/2^#3}
\let\x\pgfmathresult
\pgfmathparse{#2/2^#3}
\let\y\pgfmathresult
\pgfmathparse{1/2^#3}
\let\s\pgfmathresult
\draw[thin] (\x,\y) rectangle +(\s,\s);
}

    \draw (0,0) -- (1,0); 
    \draw (0,0) -- (0,1); 
    \draw[ dashed] (1,0) -- (1,1); 
    \draw[ dashed] (0,1) -- (1,1); 




    \draw[thin] plot[smooth cycle, tension=0.8] 
        coordinates {(0.1,0.5) (0.3,0.8) (0.6,0.9) (0.9,0.7) (0.7,0.2) (0.4,0.1)};

    \foreach \i/\j/\n in {
        1/1/2, 1/2/2, 2/2/2, 
        1/4/3, 3/1/3, 4/1/3, 4/2/3, 4/3/3, 4/6/3, 5/3/3, 6/5/3,
        2/7/4, 3/5/4, 3/6/4, 3/7/4, 3/10/4, 5/3/4, 6/12/4, 7/12/4,
        7/13/4, 10/3/4, 10/4/4, 10/5/4, 10/12/4, 10/13/4, 11/5/4, 11/12/4, 
        12/6/4, 12/7/4, 12/8/4, 12/9/4, 12/12/4, 13/9/4,
        5/12/5, 5/13/5, 5/20/5, 7/9/5, 7/22/5, 8/7/5, 9/6/5, 9/7/5, 9/24/5,
        10/24/5, 11/5/5, 11/24/5, 11/25/5, 12/26/5, 13/26/5, 14/3/5, 15/3/5,
        16/3/5, 17/3/5, 20/5/5, 22/8/5, 22/9/5, 22/26/5, 22/27/5, 23/9/5, 23/26/5,
        24/11/5, 24/26/5, 26/14/5, 26/15/5, 26/16/5, 26/17/5, 26/24/5, 26/25/5, 27/17/5,
        7/30/6, 7/31/6, 7/32/6, 7/33/6, 7/34/6, 7/35/6, 7/36/6, 7/37/6, 8/27/6,
        9/25/6, 9/26/6, 9/27/6, 9/40/6, 9/41/6, 10/23/6, 10/42/6, 11/21/6, 11/22/6, 
        11/23/6, 11/42/6, 11/43/6, 12/19/6, 12/44/6, 13/18/6, 13/19/6, 13/44/6, 13/45/6, 
        14/17/6, 14/46/6, 15/16/6, 15/17/6, 15/46/6, 17/48/6, 19/50/6, 20/11/6, 20/50/6, 
        21/10/6, 21/11/6, 21/50/6, 21/51/6, 22/9/6, 22/52/6, 23/8/6, 23/9/6, 23/52/6, 
        23/53/6, 25/7/6, 25/54/6, 26/7/6, 26/54/6, 27/7/6, 27/6/6, 27/54/6, 27/55/6, 
        29/5/6, 29/56/6, 30/5/6, 30/56/6, 31/5/6, 31/56/6, 32/5/6, 32/56/6, 33/5/6, 
        33/56/6, 33/57/6, 34/5/6, 34/56/6, 35/5/6, 35/56/6, 36/56/6, 36/6/6, 36/7/6, 
        37/56/6, 37/7/6, 38/56/6, 38/7/6, 39/56/6, 40/9/6, 40/56/6, 42/11/6, 42/11/6, 
        44/14/6, 44/15/6, 45/15/6, 46/16/6, 46/17/6, 46/54/6, 47/54/6, 48/19/6, 48/20/6, 
        48/21/6, 48/54/6, 49/21/6, 50/22/6, 50/23/6, 50/52/6, 50/53/6, 51/52/6, 52/26/6, 
        52/27/6, 54/30/6, 54/31/6, 54/32/6, 54/33/6, 54/48/6, 54/49/6, 54/50/6, 55/33/6, 
        55/48/6, 56/36/6, 56/37/6, 56/38/6, 56/39/6, 56/40/6, 56/41/6, 56/42/6, 56/43/6, 
        56/44/6, 56/45/6, 56/46/6
    } {
        \drawSquare{\i}{\j}{\n}
    }
\end{tikzpicture}
\caption{$\Omega$ and $\cB(i)$ in $\sfU$ for $i\leq 6$}
\end{figure}
We note that
\begin{equation}\label{B's-cover}
  \bigcup_{\Gamma\in \cQ_\Omega(2^n)}\Gamma=\bigcup_{i=1}^n\bigcup_{\Gamma\in \cB_i}\Gamma.
\end{equation}
Using this and~\eqref{eq: inclusions}  we see that if $\Omega$ is $\eta$-well-shaped then 
\[
2^{-2n}\#\cB_n+2^{-2(n-1)}\#\cQ_\Omega(2^{n-1})\leq \mu(\Omega)
\]
 and combining this with~\eqref{eq:another-inequality} we see that 
\begin{equation}
\label{eq: bound B}
\#\cB_n\leq \eta 2^{n+3/2}.
\end{equation}
Using again~\eqref{B's-cover} and~\eqref{eq: inclusions} we see also that if $\Omega$ is $\eta$-well-shaped then 
\begin{equation}\label{eq:sum of Bi}
  0\leq \mu(\Omega)-\sum_{i=1}^n\sum_{\Gamma\in \cB_i}\mu(\Gamma)\leq\eta 2^{-n+1/2}.
\end{equation}

\subsection{Approximations of convex sets} We now assume that $\Omega\subseteq \sfU$ is convex. In this case the closure $\overline \Omega$, interior $\Omega^\circ$,
restrictions and extensions $\Omega_{\pm \varepsilon}$, $\widetilde \Omega_{- \varepsilon}$ are all convex. 

Also recall that the boundary of a convex set has measure zero; see~\cite[Theorem~1]{Lang}.

\begin{lem}\label{lem:convex to well} A convex set $\Omega\subseteq \sfU$ is $(4+\pi)$-well-shaped.
\end{lem} 
\begin{proof}A slight variant of this is stated without proof by Schmidt~\cite[Lemma~1]{Schm}. We prove it for convex polygons and leave the approximation of general convex sets by convex polygons to the reader (compare~\cite[Lemma~2.2]{Stephani}).   
In fact for our application to the isotropic discrepancy it is enough to study only convex polygons; see~\cite[Chapter~2, Theorem~1.5]{KuNi}.
  
Since $\Omega_\varepsilon^\pm\subseteq \sfU$ we have $\mu(\Omega_\varepsilon^\pm)\leq 1\leq (4+\pi)\varepsilon$ if $\varepsilon >1$, so we  assume that $\varepsilon \leq 1$.

Given a convex polygon $\cP$ we can attach a closed rectangle of width $\varepsilon$ to the outside and inside of each side of the polygon. 

\begin{figure}[ht]
  \minipage[t]{0.32\linewidth}
  \strut\vspace*{-\baselineskip}\newline
  \begin{tikzpicture}[scale=3.7]
    \draw (0,0) -- (1,0); 
    \draw (0,0) -- (0,1); 
    \draw[ dashed] (1,0) -- (1,1); 
    \draw[ dashed] (0,1) -- (1,1); 

      
      \draw[thick, dashed] (0.0, 0.08333333333333333) -- (0.0, 0.25);
      \draw[thick] (0.0, 0.25) -- (0.95, 0.95);
      \draw[thick, dashed] (0.95, 0.95) -- (0.8333333333333334, 0.6666666666666666);
      \draw[thick] (0.8333333333333334, 0.6666666666666666) -- (0.25, 0.0);
      \draw[thick, dashed] (0.25, 0.0) -- (0.0, 0.08333333333333333);
      
   \fill[opacity=.05] (0.0, 0.08333333333333333) -- (0.0, 0.25) -- (0.95, 0.95) -- (0.8333333333333334, 0.6666666666666666) -- (0.25, 0.0) -- cycle;
  \end{tikzpicture}
  \endminipage\hfill
  \minipage[t]{0.32\linewidth}
  \strut\vspace*{-\baselineskip}\newline
  \begin{tikzpicture}[scale=3.7]
    \draw (0,0) -- (1,0); 
    \draw (0,0) -- (0,1); 
    \draw[ dashed] (1,0) -- (1,1); 
    \draw[ dashed] (0,1) -- (1,1); 

      
       \draw[thick, dashed] (0.0, 0.08333333333333333) -- (0.0, 0.25);
       \draw[thick] (0.0, 0.25) -- (0.95, 0.95);
       \draw[thick, dashed] (0.95, 0.95) -- (0.8333333333333334, 0.6666666666666666);
       \draw[thick] (0.8333333333333334, 0.6666666666666666) -- (0.25, 0.0);
       \draw[thick, dashed] (0.25, 0.0) -- (0.0, 0.08333333333333333);
       
       \fill[opacity=.05] (0.0, 0.08333333333333333) -- (0.0, 0.25) -- (0.95, 0.95) -- (0.8333333333333334, 0.6666666666666666) -- (0.25, 0.0) -- cycle;
       
       \draw[fill=RoyalBlue, opacity=.4] (0.0, 0.08333333333333333) -- (0.0, 0.25) -- (0.04, 0.25) -- (0.04, 0.08333333333333333) -- cycle;
       \draw[fill=RoyalBlue, opacity=.4] (0.0, 0.25) -- (0.95, 0.95) -- (0.9737279615219939, 0.9177977665058652) -- (0.023727961521994002, 0.2177977665058653) -- cycle;
       \draw[fill=RoyalBlue, opacity=.4] (0.95, 0.95) -- (0.8333333333333334, 0.6666666666666666) -- (0.7963462093943448, 0.6818966588768384) -- (0.9130128760610113, 0.9652299922101717) -- cycle;
       \draw[fill=RoyalBlue, opacity=.4] (0.8333333333333334, 0.6666666666666666) -- (0.25, 0.0) -- (0.21989693221172488, 0.026340184314740726) -- (0.8032302655450583, 0.6930068509814074) -- cycle;
       \draw[fill=RoyalBlue, opacity=.4] (0.25, 0.0) -- (0.0, 0.08333333333333333) -- (0.012649110640673516, 0.12128066525535389) -- (0.26264911064067353, 0.03794733192202055) -- cycle;
      \end{tikzpicture}
  \endminipage\hfill
  \minipage[t]{0.32\linewidth}%
  \strut\vspace*{-\baselineskip}\newline
  \begin{tikzpicture}[scale=3.7]
    \draw (0,0) -- (1,0); 
    \draw (0,0) -- (0,1); 
    \draw[ dashed] (1,0) -- (1,1); 
    \draw[ dashed] (0,1) -- (1,1); 

       \draw[thick, dashed] (0.0, 0.08333333333333333) -- (0.0, 0.25);
       \draw[thick] (0.0, 0.25) -- (0.95, 0.95);
       \draw[thick, dashed] (0.95, 0.95) -- (0.8333333333333334, 0.6666666666666666);
       \draw[thick] (0.8333333333333334, 0.6666666666666666) -- (0.25, 0.0);
       \draw[thick, dashed] (0.25, 0.0) -- (0.0, 0.08333333333333333);
       
       \fill[opacity=.05] (0.0, 0.08333333333333333) -- (0.0, 0.25) -- (0.95, 0.95) -- (0.8333333333333334, 0.6666666666666666) -- (0.25, 0.0) -- cycle;;
       
       \draw[fill=red, opacity=.4] (0.0, 0.08333333333333333) -- (0.0, 0.25) -- (-0.04, 0.25) -- (-0.04, 0.08333333333333333) -- cycle;
       \draw[fill=red, opacity=.4] (0.0, 0.25) -- (0.95, 0.95) -- (0.926272038478006, 0.9822022334941347) -- (-0.023727961521994002, 0.2822022334941347) -- cycle;
       \draw[fill=red, opacity=.4] (0.95, 0.95) -- (0.8333333333333334, 0.6666666666666666) -- (0.870320457272322, 0.6514366744564949) -- (0.9869871239389886, 0.9347700077898282) -- cycle;
       \draw[fill=red, opacity=.4] (0.8333333333333334, 0.6666666666666666) -- (0.25, 0.0) -- (0.2801030677882751, -0.026340184314740726) -- (0.8634364011216085, 0.6403264823519259) -- cycle;
       \draw[fill=red, opacity=.4] (0.25, 0.0) -- (0.0, 0.08333333333333333) -- (-0.012649110640673516, 0.04538600141131278) -- (0.2373508893593265, -0.03794733192202055) -- cycle;
       
       \draw[fill=red, opacity=.4] (0.0, 0.08333333333333333) -- (-0.04, 0.08333333333333333) arc (180.0:251.56505117707798:0.04) -- cycle;
       \draw[fill=red, opacity=.4] (0.0, 0.25) -- (-0.023727961521994002, 0.2822022334941347) arc (126.3843518158359:180.0:0.04) -- cycle;
       \draw[fill=red, opacity=.4] (0.95, 0.95) -- (0.9869871239389886, 0.9347700077898282) arc (-22.380135051959563:126.3843518158359:0.04) -- cycle;
       \draw[fill=red, opacity=.4] (0.8333333333333334, 0.6666666666666666) -- (0.8634364011216085, 0.6403264823519259) arc (-41.18592516570965:-22.380135051959563:0.04) -- cycle;
       \draw[fill=red, opacity=.4] (0.25, 0.0) -- (0.2373508893593265, -0.03794733192202055) arc (-108.43494882292201:-41.18592516570965:0.04) -- cycle;
      \end{tikzpicture}
  \endminipage
  \caption{Convex polygon $\cP$, $\cP$ with inner rectangles, $\cP$ with  outer rectangles and circular sectors.}
  \end{figure}

For the inside rectangles, the convexity of $\cP$ implies that these overlap and their union covers $\Omega_\varepsilon^-$.

For the outside we also add a circular sector of radius $\varepsilon$ to each corner of the polygon connecting the two sides of the attached outer rectangles. The restriction to $\sfU$ of the union of the outer rectangles and the angular sectors covers $\Omega_\varepsilon^+$. By convexity of $\cP$ the union of the circular sectors is precisely a full circle of radius $\varepsilon$. 

Let $\ell(\cP)$ be the length of the perimeter. The above considerations implies that 
\begin{align*}
&  \mu(\cP_{\varepsilon}^{+}) \leq \ell(\cP) \varepsilon+\pi\varepsilon^2\leq \(\ell(\cP)+\pi\)\varepsilon, \\
&  \mu(\cP_{\varepsilon}^{-})\leq \ell(\cP) \varepsilon. 
  \end{align*}
Since for convex sets $\Omega_1$, $\Omega_2$ with $\Omega_1\subseteq \Omega_2$ we have $\ell(\Omega_1)\leq \ell(\Omega_2)$ (see~\cite{Stephani}) we can bound $\ell(\cP)$ by the length of the perimeter of the largest convex set in $\sfU$, which is $\sfU$ itself. Since $\sfU$ has perimeter length $\ell(\sfU) = 4$ and we have assumed $\varepsilon\leq 1$ we find that $\mu(\Omega_{\varepsilon}^{\pm})\leq (4+\pi)\varepsilon$, which gives the result in the case of convex polygons. 
 \end{proof}

\section{Bounding the sums of Kloosterman sums}

\subsection{Preliminary reduction}
\label{sec:prelim} 

By partial summation, Theorem~\ref{thm:second} follows immediately (see below) from bounds on the following normalized version:
\[
\sK^{(2)} ( N;X)
=  \sum_{\pm}\sum_{N \leq n <  2N} \left|\sum_{X \leq c < 2X} \frac{1}{c} \cK(n, \pm 1, c)\right|^2.
\]
  
\begin{prop}\label{prop:Kpm} 
We have 
\[\sK^{(2)} ( N;X) \preccurlyeq   N + (NX)^{1/3} .
\] 
\end{prop}

We prove Proposition~\ref{prop:Kpm} in Section~\ref{sec:proof} after establishing several 
auxiliary results.  However, before doing so we derive Theorem~\ref{thm:second} from it, from which in turn we then 
derive Theorems~\ref{thm:Triple Sum}.  

Partitioning the interval $[1,X]$ into dyadic intervals, we re-write the bound of  Proposition~\ref{prop:Kpm} 
as 
\begin{equation}\label{eq:ModifBound}
 \sum_{\pm}\sum_{N \leq n <  2N} \left|\sum_{ c \le X} \frac{1}{c} \cK(n, \pm 1, c)\right|^2
 \preccurlyeq   N + (NX)^{1/3} .
\end{equation}
Writing, 
\begin{align*}
\sum_{ c \le X}  \cK(n, \pm 1, c)& = \sum_{ c \le X} c \cdot \frac{1}{c} \cK(n, \pm 1, c) \\
& = X \sum_{ c \le X} \frac{1}{c} \cK(n, \pm 1, c) - \int_0^X \sum_{ c \le y} \frac{1}{c} \cK(n, \pm 1, c) dy
\end{align*}
and using the Cauchy--Schwarz inequality, we conclude
\[
 \left|\sum_{ c \le X} \cK(n, \pm 1, c)\right|^2 \ll    X^2 \left|\sum_{ c \le X} \frac{1}{c} \cK(n, \pm 1, c)\right|^2+ X \int_0^X \left|\sum_{ c \le y} \frac{1}{c} \cK(n, \pm 1, c)\right|^2 dy. 
\]
Summing over $n$ with $N \leq n <  2N$ and changing the order of summation and integration, we
see that the bound~\eqref{eq:ModifBound} implies  Theorem~\ref{thm:second}. 

Next, to derive Theorem~\ref{thm:Triple Sum} from Theorem~\ref{thm:second}, we note that 
by the Cauchy--Schwarz inequality, we have 
\[
\sum_{\pm}\sum_{N \leq n <  2N} \left|\sum_{X \leq c < 2X}   \cK(n, \pm 1, c)\right| \preccurlyeq NX + N^{2/3} X^{7/6}.
\]
Now Selberg's celebrated identity (see~\cite{Sel, HaKa, Ku, Xi, And}) implies
\begin{align*}
\sum_{\pm}\sum_{M \leq m <  2M}&\sum_{N \leq n <  2N} \left|\sum_{X \leq c < 2X}   \cK(n,\pm m, c)\right| \\
&= \sum_{\pm}\sum_{M \leq m <  2M}\sum_{N \leq n <  2N} \left|\sum_{X \leq c < 2X} \sum_{d \mid \gcd(m, n, c)} 
d\cK\(\frac{mn}{d^2},\pm 1, \frac{c}{d}\)\right|\\
& \leq  \sum_{\pm}\sum_{d \leq 2\min\{M, N, X\}} \sum_{M/d \leq m <  2M/d}\sum_{N/d \leq n <  2N/d} \left|\sum_{X/d \leq c < 2X/d}    \cK(mn,\pm 1, c)\right|.
\end{align*}
Using the classical bound on the divisor function,  see~\cite[Equation~(1.81)]{IwKow}, we now write  
\begin{align*}
\sum_{\pm}\sum_{M \leq m <  2M}&\sum_{N \leq n <  2N} \left|\sum_{X \leq c < 2X}   \cK(n,\pm m, c)\right| \\
& \leq  \sum_{\pm}\sum_{d \leq 2\min\{M, N, X\}} \sum_{MN/d^2 \leq r  <  2MN/d^2} \left|\sum_{X/d \leq c < 2X/d}    \cK(r,\pm 1, c)\right|\\
& \preccurlyeq \sum_{d \leq 2\min\{M, N, X\}}  \frac{MN}{d^2} X + \frac{(MN)^{2/3} X^{7/6}}{d^{3/2}} \ll MNX + (MN)^{2/3} X^{7/6}, 
\end{align*}
which is the bound of Theorem~\ref{thm:Triple Sum}. 

 
\subsection{Analysis}
Throughout this section, we always write $s \in \C$ as $s = \sigma + i t$ with $\sigma, t \in \R$. 
Furthermore, all implied constants in the symbol `$\ll$'  may depend on $\sigma$ 
(but not on $t$), on some constants $A$, $B$ and $\Aprime$, on the integer $j$ and on $\varepsilon > 0$, wherever applicable.

For  $0 < \Delta \leq 1/4$, let $W_{\Delta}$ be a smooth function on $\R_{>0} = (0, \infty)$with $W_{\Delta} = 1$ on $[1, 2]$, $\text{supp}(W_{\Delta}) \subseteq [1-\Delta, 2 + \Delta]$ and such that $\| W^{(j)} \|_{\infty} \ll \Delta^{-j}$ for all $j \in \N \cup \{0\}$ (we recall that the implied
constant may depend on $j$).  One shows, using integration by parts and elementary estimates, that  the Mellin transform 
\begin{equation}\label{eq:mellin-definition}
\widehat{W}_{\Delta}(s)=\int_0^\infty W_{\Delta}(x)x^{s}\frac{dx}{x}
\end{equation}
 of $W_{\Delta}$ is entire and satisfies
\begin{equation}\label{w-mellin}
\widehat{W}_{\Delta}(s) \ll \frac{1}{1+|s|} \(1 + |s| \Delta\)^{-A}
\end{equation}
for every $A > 0$. The key point is that $\widehat{W}_{\Delta}$ has almost uniformly bounded $L^1$-norm on fixed vertical  strips. 

Let 
\begin{equation}\label{eq:Sets Spm}
\cS^+ = i\R \cup \(\N - \tfrac{1}{2}\) \mand \cS^- =  i\R.
\end{equation}

For $\varpi \in \cS^{\pm}$ we define
\[\cJ^{\pm}_{\varpi}(x) =  \frac{J^{\pm}_{2\varpi}(x) - J^{\pm}_{-2\varpi}(x)}{\sin(\pi \cdot \varpi)},\]
where $J^+ = J$ is the Bessel $J$-function, and $J^- = I$ is the Bessel $I$-function. 
Note that for $\varpi = k - 1/2  \in \N - \tfrac{1}{2}$, we have 
\[
\cJ^+_{\varpi}(x) = 2(-1)^{k+1} J_{2k-1}(x),
\]
 and for $\varpi = it \in i \R$ we have 
\[
\cJ^-_{\varpi}(x) = - \frac{4}{\pi} K_{2it}(x) \cosh(\pi t).
\]

We now record the Mellin transforms $\widehat{\cJ}_{\varpi}^\pm$ of $\cJ^\pm_{\varpi}$  
(see, for instance,~\cite[Equations~6.561.14 and~6.561.16]{GR} and also~\cite[Equation~(4.7)]{BB})
\[
\widehat{\cJ}^+_{\varpi}(s) 
 = \begin{cases} \displaystyle{- \frac{2^s}{\pi} \Gamma\(\frac{s}{2} + it\)\Gamma\(\frac{s}{2} - it\)\cos\(\pi s/2\)}, & \varpi = it,\\
 \displaystyle{ (-1)^{k+1}2^{s}\frac{\Gamma\(k - \frac{1}{2} + \frac{s}{2}\)}{\Gamma\(k + \frac{1}{2} - \frac{s}{2}\)}}, 
& \varpi = k - \frac{1}{2} ,\end{cases}
\]
and
\[\widehat{\cJ}^-_{\varpi}(s) 
 = - \frac{2^s}{\pi} \Gamma\(\frac{s}{2} + it\)\Gamma\(\frac{s}{2} - it\)\cosh(\pi t), \qquad \varpi = it.
\]
For a smooth function with compact  support on $\R_{>0}$ and $\varpi \in \cS$ we define
\begin{equation}\label{w-hat pm}
 \widecheck{V}^{\pm}(\varpi) = \int_0^{\infty} V(x) \cJ^{\pm}_{\varpi}(x) \frac{dx}{x}. 
\end{equation}

\begin{lem}\label{analysis} Let  $r \in \R$, $Z > 0$, $\varpi \in \cS^{\pm}$ and $V$ be a fixed smooth function with compact support in $\R_{>0}$. Let
\[\cI^{\pm}(r, \varpi, Z) = \int_0^{\infty} V\(\frac{x}{Z}\)x^{ir} \cJ^{\pm}_{\varpi}(x) \frac{dx}{x}.\]
Then, for any constant $A > 0$, 
\begin{itemize}
\item if $\varpi = it \in i\R$, then
 \begin{align*} 
\cI^{\pm}&\(r, it, Z\) \\
& \ll   \frac{1}{\(\(1 + \left||r| - 2|t|\right|\)\(1 + |r|+2|t|\)\)^{1/2}}\(1 + \frac{1+\left||r| - 2|t|\right|}{1 + \frac{Z^2}{1 + |r| + 2|t|}}\)^{-A};
\end{align*} 
\item if $\varpi = k - \tfrac{1}{2} \in \N - \tfrac{1}{2}$,  then
\[
\cI^{+}\(r, k-\tfrac{1}{2}, Z\)  \ll  \frac{1}{k}\(1 + \frac{k}{1+Z}\)^{-A}.
\]
\end{itemize}
\end{lem}

\begin{proof} By the Mellin inversion formula we have
\[\cI^{\pm}(r, \varpi, Z) = \int_{\sigma + i \R} \widehat{\cJ}^{\pm}_{\varpi}(s) \widehat{V}(ir - s) Z^{ir - s} \frac{ds}{2\pi i}\]
for $\sigma > 0$. Note that $\widehat{V}$ is entire and rapidly decaying on vertical lines. 

In the case $\varpi = k - \tfrac{1}{2}$ we shift the contour to $\sigma  = 0$ to obtain the bound 
\begin{align*} \cI^+(r, k - \tfrac{1}{2}, Z) & 
\ll \int_{\R}\left|\frac{\Gamma(k - \frac{1}{2} + \frac{i\tau}{2})}{\Gamma(k + \frac{1}{2} - \frac{i\tau }{2})}\right|
 \(1 + |\tau - r|\)^{-10} d\tau \\
& \ll \frac{1}{k} \int_{\R}  (1 + |\tau - r|)^{-10} d\tau \ll \frac{1}{k}.
\end{align*} 
 Let $A > 0$. If $k \leq A/2 + 1$, there is nothing else to show. Otherwise, we shift the contour to $\sigma = -A$ without crossing poles, using a bound similar to~\eqref{w-mellin}, we obtain   
 \begin{align*} 
 \cI^+\(r, k - \tfrac{1}{2}, Z\) & \ll Z^A  \int_{\R} 
 \left|\frac{\Gamma(k - \frac{1+A}{2} + \frac{i\tau}{2})}
 {\Gamma(k + \frac{1+A}{2} - \frac{i\tau }{2})}\right| (1 + |\tau - r|)^{-10} d\tau\\
 & \ll  Z^A k^{-1-A}.
\end{align*} 
 This completes the proof in the case $\varpi = k - \tfrac{1}{2}$. 

In the case $\varpi = it$ we shift the contour  to $\sigma = -\Aprime \not\in \{0, -2, -4, \ldots\}$ with $\Aprime>  -1/2$. 
Then, for any $B > 0$, using a bound similar to~\eqref{w-mellin} and trivial bounds on $\cos\(\pi s/2\)$ and 
$\cosh(\pi t)$, we see that   the corresponding integral $\fI$ along the vertical line $\RE s = -\Aprime$ is bounded by  
\begin{align*} 
\fI  &\ll Z^{\Aprime} \int_{\R} \Gamma\(-\frac{\Aprime}{2} + i \(\frac{\tau}{2} + t\)\)  \Gamma\(-\frac{\Aprime}{2} + i \(\frac{\tau}{2} - t\)\) \\
& \qquad  \qquad  \qquad  \qquad  \qquad \qquad \times  \( e^{\frac{1}{2}\pi |\tau|} +e^{\pi |t|} \)  \(1 + |\tau - r|\)^{-B} d\tau \\
& \quad \ll  Z^{\Aprime} \int_{\R} \((1 + |\tau - 2t|)(1 + |\tau + 2t|)\)^{-(\Aprime+1)/2 } (1 + |\tau - r|)^{-B} d\tau,
\end{align*} where we have used the Stirling asymptotics on the Gamma functions.
Since
 $(1 + |\tau \pm 2t|) (1 + |\tau - r|) \geq (1 + |r \pm 2t|)$, we can further estimate this as 
 \begin{equation}\label{bound1}
\begin{split}
\fI &  \ll Z^{\Aprime}  ((1 + |r - 2t|)(1 + |r + 2t|)^{-(\Aprime+1)/2 } 
\int_{\R} (1 + |\tau - r|)^{-(B- \Aprime-1)} d\tau\\
 & \ll \(\frac{Z^2}{(1 + |r - 2t|)(1 + |r + 2t|)}\)^{\Aprime/2} 
\frac{1}{((1 + |r - 2t|)(1 + |r + 2t|))^{1/2 }}, 
 \end{split}
\end{equation}
provided that $B \geq \Aprime+3$, say.

If $\Aprime > 0$, then  for any $B > 0$ we also pick up residues coming from the poles of the Gamma factors. These are bounded by
\begin{align*} 
\fR & \ll \sum_{\pm}\sum_{n = 0}^{[\Aprime/2]} \frac{Z^{2n}}{(1 + |t|)^{n + 1/2}}(1 + |r \pm 2t|)^{-B}\\
& \ll  \sum_{\pm} \frac{(1 + |r \pm 2t|)^{-B}}{(1 + |t|)^{1/2}}\(1 +   \frac{Z^{\Aprime}}{(1 + |t|)^{\Aprime/2}}\).
\end{align*} 
It is easily seen that $(1 + 4|t|)(1 + |r\pm 2t|) \geq (1 + |r \mp 2t|)$, so that we have 
\begin{align*} 
\fR  & \ll  \sum_{\pm} \frac{(1 + |r \pm 2t|)^{-(B - \frac{\Aprime+1}{2})}}{(1 + |r \mp 2t|)^{1/2}}\(1 +   \frac{Z^{\Aprime}}{(1 + |r \mp 2t|)^{\Aprime/2}}\)\\
&= \sum_{\pm} \frac{(1 + |r \pm 2t|)^{-(B-1 - \frac{\Aprime}{2})}}{((1 + |r - 2t|)(1 + |r + 2t|))^{1/2}}\(1 +   \frac{Z^{\Aprime}}{(1 + |r \mp 2t|)^{\Aprime/2}}\).
\end{align*} 
Assuming that $B - 1 - \Aprime/2 \geq  \Aprime/2$ and noting that $\{|r + 2t|, |r - 2t|\} = \{\left|| r | - 2 |t|\right|, |r| + 2 |t|\}$, 
this is yields 
\begin{align*} 
\fR & \ll   \frac{1}{((1 + |r - 2t|)(1 + |r + 2t|))^{1/2}}\\
& \qquad\times \((1 +\left|| r | - 2 |t|\right| )^{-D}+  \( \frac{Z^{2}}{(1 + ||r| -  2|t||)(1 + |r| + 2 |t|)} \)^{D}\) 
\end{align*} 
with $D = \Aprime/2$, 
which then simplifies as 
\begin{equation}\label{bound2}
\fR \ll    \frac{1}{((1 + |r - 2t|)(1 + |r + 2t|))^{1/2}}\(\frac{1 +\left|| r | - 2 |t|\right| }{1 + \frac{Z^2}{1 + |r| + 2 |t|}}\)^{-D}. 
\end{equation}   
Thus we write 
\[
\cI^{\pm}\(r, it, Z\)  \le |\fI| + |\fR|
\]
and  apply~\eqref{bound1} and~\eqref{bound2} with $\Aprime = 2A$, and thus $D = A$ 
(observe that~\eqref{bound2}   dominates~\eqref{bound1}), and we apply~\eqref{bound1} with $\Aprime= -1/4$, say (in which case~\eqref{bound2} is void) to obtain a slightly stronger result than claimed. This completes the proof  in the case $\varpi = it$. 
\end{proof}

\subsection{The Kuznetsov formula and the spectral large sieve}
We now state the Kuzne\-tsov formula in an abstract form with the spectral large sieve incorporated. 
Recall the definitions~\eqref{eq:Sets Spm} and~\eqref{w-hat pm}.

\begin{lem}\label{lem2a} Let $m, n \in \N$ and $V$ a smooth  function with compact support on $\R_{>0}$. Then there exists a measure $d\mu(\varpi)$ on $\cS^{\pm}$ and real numbers $\rho_{\varpi}(n), \rho_{\varpi}(m)$ such that 
\[\sum_{c=1}^\infty \frac{1}{c} \cK(m, \pm n, c) V\( 4\pi\frac{\sqrt{mn}}{c}\) = \int_{\cS^{\pm}} \rho_{\varpi}(m) \rho_{\varpi}(n)  \widecheck{V}^{\pm}(\varpi) d\mu(\varpi), \]
where  
for  any $T, N \geq 1$ and $1 \leq H \leq T$ and any sequence $\balpha = \(\alpha(n) \)_{n=1}^N \subseteq \C$   we have 
\[\int_{T-H \leq |\varpi| \leq T} \left| \sum_{n \leq N} \rho_{\varpi}(n) \alpha(n)\right|^2 d\mu(\varpi) \ll \(TH + N\)^{1+\varepsilon} \| \balpha \|_2^2, \]
where $ \| \balpha \|_2$ is the $\ell^2$-norm of $\balpha$. 
\end{lem}

The Kuznetsov formula is stated in~\cite[Theorem~16.5 and~16.6]{IwKow}. The measure $d\mu(\varpi)$ has a discrete part (holomorphic and non-holo\-morphic cusp forms) and a continuous part (Eisenstein series).  Note that for the group $\SL_2(\Z)$ there are no exceptional eigenvalues. 

The large sieve inequality for $H = T$ is a famous result of Deshouillers and Iwaniec~\cite{DI}; the hybrid form for general $H$ can be found in~\cite[Theorem~1.1]{Ju} in the non-holomorphic case. The proof (cf.~\cite[Equation~(3.1)]{Ju}) shows that it holds automatically for the continuous spectrum. The holomorphic case is similar, and can be found, for instance, in~\cite{La}.  In fact, the work of Lam~\cite{La}  includes the hybrid form for general $H$ in both  non-holomorphic and holomorphic cases.

We quickly derive the following ``dual'' version of the spectral large sieve. 

\begin{lem}\label{lem3a}
In the notation of Lemma~\ref{lem2a}, for any 
  measurable function $f : \cS \rightarrow \C$ we have
\[
\sum_{n\le N}   \left| \int_{T-H \leq |\varpi| \leq T} \rho_{\varpi}(n) f(\varpi)d\mu(\varpi) \right|^2 
 \ll  \(TH + N\)^{1+\varepsilon}  \int_{T-H \leq |\varpi| \leq T} |f(\varpi)|^2 d\mu(\varpi).  
\]
\end{lem}

\begin{proof} Let
\[\cL= 
\sum_{n\le N} \left| \int_{T-H\leq |\varpi|\leq T} \rho_{\varpi}(n) f(\varpi)d\mu(\varpi)\right|^2  =  \int_{T-H \leq |\varpi| \leq T} f(\varpi) \sum_{n \leq N} \rho_{\varpi}(n)  \alpha_n\, d\mu(\varpi),
\]
where
\[\alpha_n =  \int_{t-H \leq |\sigma| \leq T} \overline{f(\sigma)} \rho_{\sigma}(n)\, d\mu(\sigma). \]
By the Cauchy--Schwarz inequality and the large sieve inequality of Lemma~\ref{lem2a}, we obtain
\[\cL^2\ll    \int_{T-H \leq |\varpi| \leq T} |f(\varpi)|^2 d\mu(\varpi) \, \cdot \, (TH + N)^{1+\varepsilon} \sum_{n\le N}  |\alpha_n|^2\] 
and since 
\[
\sum_{n\le N} |\alpha_n|^2 = \cL ,
\] 
the claim now follows. 
\end{proof} 

The next estimate follows immediately from Lemma~\ref{lem3a} using the function $\varpi\mapsto f(\varpi)\rho_\omega(1)$, followed by Lemma~\ref{lem2a} with $N=1$.

\begin{cor}\label{coro}
Let $N, T \geq 1$, $F \geq 0$ and $1 \leq H \leq T$. Let  $f : \cS \rightarrow \C$ be a measurable function with $|f(\varpi)| \leq F$ for $T-H \leq |\varpi| \leq T$. Then 
\[
\sum_{n \leq N} \left| \int_{T-H \leq |\varpi| \leq T} \rho_{\varpi}(n) \rho_{\varpi}(1) f(\varpi)d\mu(\varpi) \right|^2 \ll F^2 TH \(TH + N\)^{1+\varepsilon}. 
\]
\end{cor}

\subsection{Proof of Proposition~\ref{prop:Kpm}}
\label{sec:proof}

 We define
\[V_{n, \Delta}(x) = W_{\Delta}\(x \frac{X}{4\pi \sqrt{n}}\),
\]
where $W_{\Delta}$ is as in~\eqref{w-mellin} and the surrounding discussion, and we write
\[
\sK^{(2)}_{\ast}  (N;X)  =   \sum_\pm\sum_{N \leq n <  2N} \left|\sum_{c=1}^\infty \frac{1}{c} \cK(n, \pm 1, c)V_{n, \Delta}\(4\pi  \frac{\sqrt{n}}{c}\)\right|^2.
\]
We assume generally that $\Delta \geq (NX)^{-100}$. 
By  the Weil bound~\eqref{eq: Kloost},  and using $\mathbf{1}_{X\leq c<2X}$ for the characteristic function of the interval $[X,2X)$, it is easy to see that 
\begin{align*}
 \left|\sum_{X\leq c < 2X}   \frac{1}{c} \cK(n, \pm 1, c)\right|^2
&  \le 2  \left|\sum_{c=1}^{\infty}   \frac{1}{c} \cK(n, \pm 1, c) \(\mathbf{1}_{X\leq c<2X}- V_{n, \Delta}\(4\pi  \frac{\sqrt{n}}{c}\)\) \right|^2\\
& \qquad \qquad \qquad + 2  \left|\sum_{c=1}^{\infty}   \frac{1}{c} \cK(n, \pm 1, c)  V_{n, \Delta}\(4\pi  \frac{\sqrt{n}}{c}\)\ \right|^2\\
&  \preccurlyeq   \Delta^2 X + \left|\sum_{c=1}^\infty \frac{1}{c} \cK(n, \pm 1, c)V_{n, \Delta}\(4\pi  \frac{\sqrt{n}}{c}\)\right|^2 . 
\end{align*}  
Therefore, 
\begin{equation}\label{weil}
\sK^{(2)}  (N;X)  \preccurlyeq  N\Delta^2 X+  \sK^{(2)}_{\ast} (N;X) .
\end{equation} 
We proceed to estimate $\sK^{(2)}_{\ast} (N;X)$. To this end, let 
\[V(x) = W\(x \frac{X}{4\pi \sqrt{N}}\), \]
where $W$ is a fixed smooth function with support in $[1/100, 100]$ and $W(x) = 1$ on $[1/50, 50]$. We   insert this as a redundant factor and write
\[
\sK_*^{(2)} (N;X)= \sum_{\pm} \sum_{N \leq n <  2N}   \left|\sum_{c=1}^\infty \frac{1}{c} \cK(n, \pm 1, c)V_{ n, \Delta}\(4\pi  \frac{\sqrt{n}}{c}\)V\(4\pi  \frac{\sqrt{n}}{c}\)\right|^2.
\]
The inner sum is in shape to apply the Kuznetzov formula given by Lemma~\ref{lem2a},  
  and   we obtain  
\[
\sK^{(2)}_* (N;X) = \sum_{\pm} \sum_{N \leq n <  2N}    \left|  \int_{\cS^{\pm}}  \rho_{\varpi}(1) \rho_{\varpi}(n) \widecheck{(V_{n, \Delta} V)}^{\pm}(\varpi) d\mu(\varpi) \right|^2.
\]
We define 
\[
U(r,\varpi)^{\pm}  =   \int_0^{\infty} V(x) x^{ir} \cJ^{\pm}_{\varpi}(x) \frac{dx}{x}
\]
and 
Mellin-invert $V_{n, \Delta}$, getting  
\begin{align*}
 \sK^{(2)}_{\ast} ( N;X)& \ll \sum_{\pm }  \sum_{N \leq n <  2N}\left|  \int_{\R}\frac{(1 + |r|\Delta)^{-10}}{1+|r| }   \left|  \int_{\cS^{\pm}}  \rho_{\varpi}(1) \rho_{\varpi}\(n\)
U(r,\varpi)^{\pm}
 d\mu(\varpi) \right| dr \right|^2\\
 &  \preccurlyeq \sum_{\pm }  \sum_{N \leq n <  2N}  \int_{\R}\frac{(1 + |r|\Delta)^{-10}}{1+|r| }   \left|  \int_{\cS^{\pm}}  \rho_{\varpi}(1) \rho_{\varpi}\(n\)
U(r,\varpi)^{\pm}
 d\mu(\varpi) \right|^2 dr ,  \end{align*} 
 where the second step is an application of the Cauchy--Schwarz inequality.

We   split the set $\cS^{\pm}$ into countably many  pieces 
\[
\cS^{\pm} = \bigcup_{j=1}^\infty \cS^{\pm}_j
\] 
according to the size of $\varpi$ and also depending on $r$, to be described in a moment, and to each subset $\mathcal{S}_j^{\pm}$ we attach a weight $\omega_j^{\pm} > 0$ also to be determined in a moment. For now we just require  that the series  
\[\Omega^{\pm}(r) = \sum_{j=1}^{\infty} \omega_j^{\pm} < \infty
\]
converges, and we  choose them later in a way that 
\begin{equation}\label{small}
  \Omega^{\pm}(r) \ll ((1+|r|) NX)^{\varepsilon}.
\end{equation}
with the implied constant depending only on an arbitrary small   $\varepsilon >0$.
Again by the Cauchy--Schwarz inequality, we obtain
\begin{equation}\label{kast}
\sK^{(2)}_{\ast} (N;X)   \preccurlyeq \sum_{\pm} \int_{\R}\frac{(1 + |r|\Delta)^{-10}}{1+|r| }  \Omega^{\pm}(r) \Sigma_d^\pm(r)  dr, 
\end{equation} 
where 
\[
 \Sigma_d^\pm(r) = \sum_{j=1}^\infty (\omega_j^{\pm})^{-1}  \sum_{N \leq n < 2N}  \left|  \int_{\cS^{\pm}_j}  \rho_{\varpi}(1) \rho_{\varpi}(n) U(r,\varpi)^{\pm}  
 d\mu(\varpi) \right|^2 . 
\]

This expression is in shape for an application of Corollary~\ref{coro}. 
The integral transform $ U(r,\varpi)^{\pm} $
 has been estimated in Lemma~\ref{analysis}, and accordingly we distinguish the cases 
\[
\varpi = k - \tfrac{1}{2} \in \N - \tfrac{1}{2} \mand \varpi = it \in i\R.
\]

We start with the case $\varpi = k - \tfrac{1}{2}$ and split into dyadic ranges $T_{\nu} \leq k < 2 T_{\nu}$, where 
$T_\nu = 2^\nu$, $\nu = 0,1, \ldots $. 
Using Lemma~\ref{analysis} with  
$A = 10$ and $Z = \sqrt{N}/X$ and Corollary~\ref{coro} with $T/2=H=T_\nu$  and 
\[
F \ll \frac{1}{T_\nu}\frac{1}{(1+T_\nu/(\sqrt{N}/X))^{10}},
\] 
we  bound, up to $(TN)^{o(1)}$, this portion  of $  \Sigma_d^+(r)$ by 
\begin{equation}\label{eq:bound k-1/2}
\sum_{\nu=0}^{\infty}(\omega_{\nu}^+)^{-1} \frac{1}{T_\nu^2} \(1 + \frac{T_\nu}{\sqrt{N}/X}\)^{-20} T_\nu^2(T_\nu^2 + N) \preccurlyeq N
\end{equation}
upon choosing $\omega_{\nu}^+ =  \(1 + \frac{T_\nu}{\sqrt{N}/X}\)^{-10}$, say, which satisfies the condition~\eqref{small}. 

Next we consider the case $\varpi = it \in i\R$ using again Lemma~\ref{analysis} with  $Z = \sqrt{N}/X$. We split into smaller windows $H_{\nu} \leq 1 + \left||r| - 2|t|\right | <  2H_\nu$ where 
\[
H_\nu  = 2^\nu \mand 
T_\nu = H_\nu + |r|, \qquad \nu = 0,1, \ldots.
\]
In this region we have 
\[
t\leq T_\nu \mand  1/(1+|r|+|2t|)\leq 2/T_\nu
\] 
and  
the set of $t$ satisfying these conditions is contained in the union of at most 4 windows of size at most $H_\nu$.
 With $Z = \sqrt{N}/X$ and 
 \[
 F \ll \frac{1}{(T_vH_v)^{1/2}}\left(1+\frac{H_\nu}{1+2Z^2/T_\nu}\right)^{-10}
 \] we bound as before  (and again up to  $(TN)^{o(1)}$)  this portion  
 of $  \Sigma_d^{\pm}(r)$ by
\begin{equation}\label{bound it}
\begin{split}
 \sum_{\nu=0}^\infty  &(\omega_{\nu}^{\pm})^{-1}  \frac{1}{T_\nu H_\nu }  \(1 + \frac{H_\nu }{1+Z^2/T_\nu }\)^{-20} T_\nu H_\nu (T_\nu H_\nu+ N) \\
&\ll   \sum_{\nu=0}^\infty    \(1 + \frac{H_\nu}{1+Z^2/(H_\nu+ |r|)}\)^{-10}  \((H_\nu+|r|)H_\nu +  N \)  \\
& \preccurlyeq    \(1 + |r|^{1+\varepsilon} + Z^2 + N\) \ll  1 + |r|^{1+\varepsilon} + N
\end{split}
\end{equation} 
upon choosing $\omega_{\nu}^{\pm} =    \(1 + \frac{H_\nu }{1+Z^2/T_\nu }\)^{-10}$, say, which satisfies~\eqref{small}
(we still allow the implied constants to depend on $\varepsilon>0$).

We substitute~\eqref{eq:bound k-1/2} and~\eqref{bound it} back into~\eqref{kast} getting
\[
\sK^{(2)}_{*} ( N;X) 
 \preccurlyeq   \int_{\R}\frac{\(1 + |r|\Delta\)^{-10}}{1+|r| } (1 + |r|^{1+\varepsilon} + N)dr \preccurlyeq N + \Delta^{-1}. 
\]

Together with~\eqref{weil} we finally obtain
\[
\sK^{(2)}  (N;X)  \preccurlyeq N+\Delta^{-1} + N\Delta^2 X\\
  \ll  N +  (NX)^{1/3} 
\]
upon choosing $\Delta = (NX)^{-1/3}$, which balances the two last terms. This finishes the proof of Proposition~\ref{prop:Kpm}
and thus of Theorem~\ref{thm:second},  and as such also of Theorem~\ref{thm:Triple Sum}.

\section{Bounding the uniformity of distribution measures}

\subsection{Proof of Theorem~\ref{thm:Discr}}   
From~\eqref{eq: NX-total} and Lemma~\ref{lem:K-S} we see, using a dyadic decomposition, that for any integer $1 \le \ell \ll \log X$ 
we have
\begin{equation}
\label{eq: DX SS}
\Delta(X, \cB)  \ll  2^{-\ell} +   \frac{1}{X^2}  \(S_*+S_0\), 
\end{equation}
where 
\[
 S_* = \sum_{i,j=0}^\ell \fK(2^i, 2^j, X) 2^{-i-j} \mand
 S_0= \sum_{i=0}^\ell  \sum_{2^i \le m <2^{i+1}} \left|\sum_{c  \le X} \cK(0, m; c) \right| 2^{-i} .
\]
 
 By Theorem~\ref{thm:Triple Sum} and Remark~\ref{rem:sumRaman} we have 
\[\Delta(X, \cB)\preccurlyeq  2^{-\ell}+X^{-2}(\ell^2X+ X^{7/6}).
\]
Taking $\ell = \rf {\log X/\log 2} $ in~\eqref{eq: DX SS} we   conclude the proof of Theorem~\ref{thm:Discr}.

\subsection{Proof of Theorem~\ref{thm:SmallBox}}
 We choose two  parameters $M,N \in \mathbb{N}$ such that 
\begin{equation}
\label{eq: Large M N}
\alpha M, \beta N \ge 2,
 \end{equation}
 where $\alpha$ and $\beta$ are the side lengths of the box $B$ as in~\eqref{eq: Box B}.

It   follows from  Theorem~\ref{thm:Triple Sum}   (and Remark~\ref{rem:sumRaman}),  equation~\eqref{eq: NX-total} and Lemma~\ref{lem:BMW-corollary} 
that 
\[
\#(B\cap \cS(X)) = \mu\(B\)N(X)\(1 + O(E)\), 
\]
where 
\begin{equation}
\label{eq: Delta}
E \preccurlyeq  \alpha^{-1}  M^{-1}  + \beta^{-1} N^{-1}  + MNX^{-1} + (MN)^{2/3} X^{-5/6}.
\end{equation}
We balance the first two terms and choose $N =\rf{\alpha  \beta^{-1} M}$.   
We observe that to satisfy~\eqref{eq: Large M N} we now only need to 
request $\alpha  M\ge 2$.

We can certainly assume that $M,N \le X^2$ as otherwise
 the result is   trivial.
Hence we see from~\eqref{eq: Delta}   that 
\[
\Delta  \preccurlyeq \alpha^{-1} M^{-1}  + \alpha  \beta^{-1} M^2 X^{-1} +  
\(\alpha  \beta^{-1}\)^{2/3}M^{4/3}X^{-5/6} . 
\]

 Clearly the above expression minimises at one of the two values of $M$
balancing the first and either the second or the third term. 
That is,  at 
\[
M_1 = \rf{2\alpha^{-2/3} \beta^{1/3}  X^{1/3}}, 
\]
or at 
\[
M_2 = \rf{2\alpha^{-5/7} \beta^{2/7} X^{5/14}}. 
\]  
Choosing  $M=M_1$ leads to the bound
\begin{align*} 
  \Delta  &\preccurlyeq  \alpha^{-1} M_1^{-1} + \(\alpha  \beta^{-1}\)^{2/3}M_1^{4/3}X^{-5/6} \\
  &\preccurlyeq \mu\(B\)^{-1/3} X^{-1/3} +   \mu\(B\)^{-2/9} X^{-7/18}  .
  \end{align*}  
  We easily verify that for $\mu\(B\)\leq 1\leq X$ the first term dominates the second, 
which finishes the proof of Theorem~\ref{thm:SmallBox}.

\subsection{Proof of Theorem~\ref{thm:SmallDisc}}
Using Lemma~\ref{lem:Harm} with $d=2$ together with   Theorem~\ref{thm:Triple Sum} and a dyadic partition of unity as in the proof of Theorem \ref{thm:Discr}, we conclude  for any integer $L$ with  $X  \ge L \ge 1$ that
\begin{align*}
&\left|  \frac{\#(D \cap \cS(X)) - \pi R^2}{N(X)} \right| \\
&\quad\preccurlyeq \max_{1 \leq M \leq N \leq L} \left\{ \frac{R}{L} + \frac{1}{L^2} + \frac{1}{X^2} \( \frac{1}{L^2} + \min\left\{R^2, \frac{R^{1/2}}{N^{3/2}}\right\}\) \(NMX + (NM)^{2/3}X^{7/6}\)\right\}\\
&\quad \preccurlyeq \max_{1  \leq N \leq L}\left\{ \frac{R}{L} + \frac{1}{L^2} + \frac{1}{X^2} \( \frac{1}{L^2} + \min\left\{R^2, \frac{R^{1/2}}{N^{3/2}}\right\}\) \(N^2X + N^{4/3}X^{7/6}\)\right\}\\
&\quad \preccurlyeq   \frac{R}{L} + \frac{1}{L^2} + \frac{1}{X} + \frac{1}{L^{2/3}X^{5/6}}+ \min\left\{R^2, \frac{R^{1/2}}{L^{3/2}}\right\}\frac{L^2}{X}  +  \min\left\{R^{2} L^{4/3}, R^{2/3}\right\}\frac{1}{X^{5/6}}. 
\end{align*}
We choose 
\[
L = \fl{X^{1/2} + X^{2/3} R^{1/3}}
\]
so that our bound becomes
\[
\left|  \frac{\#(D \cap \cS(X)) - \pi R^2}{N(X)} \right|  \preccurlyeq \begin{cases} R^{2/3}X^{-2/3} , & R \geq X^{-1/2},\\ X^{-1} , & R \leq X^{-1/2}, \end{cases}
\]
as claimed.

\subsection{Proof of Theorem~\ref{thm:Isotrop Discr}} 

Recall that we say that a set $C \subseteq \(\R/\Z\)^2$ is convex if there exist a   convex set $\Omega\subseteq \sfU = [0, 1)^2$ such that $\Omega\bmod \Z^2= C$.
 Similarly, given $\eta>0$, we say that $\Delta$ is $\eta$-well-shaped if there exist an $\eta$-well-shaped $\Omega\subseteq \sfU$ as defined in 
 Section~\ref{sec:approx-alternative} satisfying $\Omega\bmod \Z^2=\Delta$.

For the proof of Theorem~\ref{thm:Isotrop Discr} it is technically easier to establish a more general result 
 for the discrepancy 
\[
\Delta_\eta(X) =  \sup_{S,~\eta\text{-well-shaped}} \left|\frac{\#(S \cap \cS(X))}{N(X)}   - \mu(S)\right|
\]
with respect to $\eta$-well-shaped sets $S \subseteq \(\R/\Z\)^2$.
This is indeed a more general result since, by Lemma~\ref{lem:convex to well}, any convex set in $\(\R/\Z\)^2$ is $(4+\pi)$-well-shaped 
so $\Delta(X, \cC) \leq \Delta_{4+\pi}(X)$. 

We  show that 
\begin{equation}\label{well-shaped-discrepancy}
  \Delta_\eta(X) \preccurlyeq \eta X^{-11/24 }, 
\end{equation}
uniformly over $\eta$. Clearly if $\eta$ is fixed (like in the above scenario relevant to bounding 
$\Delta(X, \cC)$), the bound~\eqref{well-shaped-discrepancy} simplifies as $  \Delta_\eta(X)  \preccurlyeq  X^{-11/24}$. However, for 
completeness and keeping in mind that there could be other applications to $\eta$-well-shaped sets 
where  $\eta$ may depend on $X$, we derive the result in full generality.

To prove~\eqref{well-shaped-discrepancy} we note that for $\Omega\subseteq \sfU$ we have, by definition, that  $(\sfU\backslash \Omega)_\varepsilon^\pm=\Omega_\varepsilon^\mp$, using the notation from Section~\ref{sec:approx-alternative}.  It follows that the complement of an $\eta$-well-shaped set $S \subseteq \(\R/\Z\)^2$ is again an $\eta$-well-shaped set. We therefore conclude, using $\#(S^c \cap \cS(X)) =N(X)-\#(S \cap \cS(X))$, that in order to establish 
 Theorem~\ref{thm:Isotrop Discr} it is enough to prove a lower bound 
 \begin{equation}
\label{eq: LB NX} 
\#(S  \cap \cS(X))  \ge  \mu(S) N(X) + O\(\eta X^{37/24+o(1)}\)
\end{equation}
for such sets.

We start by approximating an $\eta$-well-shaped set $S$ by boxes in $\cB_i$ as described in 
Section~\ref{sec:approx-alternative};~ cf.\  in particular~\eqref{B's-cover}. From the third inclusion in~\eqref{eq: inclusions} combined with the fact that rational points are in at most in one of the boxes $B_i$ we see that 
\[
\#(S  \cap \cS(X)) \ge \sum_{i=1}^M  \sum_{\Gamma \in \cB_i}\#(\Gamma\cap \cS(X)). 
\]
Here we interpret $\Gamma$ as having been mapped to $\(\R/\Z\)^2$.
We want to apply Theorem~\ref{thm:Discr} to the larger of these boxes, and Theorem~\ref{thm:SmallBox} to the smaller ones. We first assume  $M \leq \log X/(2\log(2))$ such that $\cB_M$ contains boxes that are allowed in Theorem~\ref{thm:SmallBox}. 
  We then choose another positive parameter $L \le M$ and apply 
Theorem~\ref{thm:Discr} to squares $\Gamma \in \cB_i$ with $i \le L$ 
and Theorem~\ref{thm:SmallBox} otherwise. In this way we derive 
 \begin{equation}
\label{eq: N partition}
\#(S  \cap \cS(X)) \ge N(X)  \sum_{i=1}^M  \sum_{\Gamma \in \cB_i} \mu (\Gamma) + O\(R X^{2+o(1)}\), 
\end{equation}
where 
\[
R =   \sum_{i\le L}  \# \cB_i  X^{-5/6} +    \sum_{L < i\le M}  \# \cB_i   \(2^{-2i}\)^{2/3}  X^{-1/3}. 
\]

Recalling the bound~\eqref{eq: bound B} on $ \# \cB_i $, we see that 
 \begin{equation}
 \label{eq: bound R}
R    \ll \eta (2^L X^{-5/6} +  2^{-L/3} X^{-1/3}) . 
\end{equation}
Furthermore, by~\eqref{eq:sum of Bi} we have 
 \begin{equation}
\label{eq: mu sum mu}
\sum_{i=1}^M  \sum_{\Gamma \in \cB_i} \mu (\Gamma) = \mu(S) + O\(\eta  2^{-M}\). 
\end{equation}
Putting these estimates together in~\eqref{eq: N partition} we find 
\[
\#(S  \cap \cS(X)) \ge N(X)\mu(S) + O(\eta X^{2+o(1)}(2^LX^{-5/6}+2^{{-L/3}}X^{-1/3}+2^{-M})).
\]

 We balance  the two first terms and choose 
 \[
 L=\fl{\frac{3\log X}{8\log 2}}.
 \]
  Then these two terms are of order $X^{37/24+o(1)}$. Choosing $M$ to satisfy 
\[
 \frac{11\log X}{24\log 2} \leq M\leq \frac{\log X}{2\log 2},
\] 
we see with this choice $L<M$ and $X^2/2^M$ is bounded by $X^{37/24+o(1)}$. This gives the desired lower bound~\eqref{eq: LB NX}, and therefore concludes the proof of~\eqref{well-shaped-discrepancy}.

 \section*{Acknowledgements} 
 
The authors are grateful to Hugh Montgomery for his suggestion to use~\cite{BMV}  and to Christoph Aistleitner and Yiannis N. Petridis 
for very useful discussions. The authors are also grateful to the referee for the very careful reading of the 
manuscript and several very helpful suggestions.

This material is based upon work supported by the Swedish Research Council under grant no.~2021-06596 while the authors were in residence at Institut Mittag-Leffler in Djursholm, Sweden, during the programme \lq Analytic Number Theory\rq{} in January-April of 2024.

During the preparation of this work   V.B.\ was supported in part by Germany's Excellence Strategy grant EXC-2047/1 -- 390685813, by SFB-TRR 358/1 2023 -- 491392403 and by  ERC Advanced Grant 101054336. 
M.R.\ was supported by grant DFF-3103-0074B from the Independent Research Fund Denmark. 
I.S.\ was  supported by the Australian Research Council Grants  DP230100530 and DP230100534 and by a 
Knut and Alice Wallenberg Fellowship.

\end{document}